\documentclass[11pt]{amsart}
\usepackage{hyperref}
\hypersetup{colorlinks = true, linkcolor = {blue}}
\usepackage[all]{xy}
\usepackage{amssymb}
\usepackage{aliascnt}
\usepackage{verbatim}
\usepackage{tikz}
\usetikzlibrary{decorations.markings}
\usepackage[colorinlistoftodos]{todonotes}
\usepackage{amsmath}
\usepackage{amsthm}
\usepackage{latexsym,amsfonts,amssymb}
\usepackage{amsmath,amsthm,amscd}
\usepackage[dvips]{epsfig}
\usepackage{psfrag}

\usepackage{color}
\usepackage{etoolbox}
\usepackage{hyperref}
\usepackage{a4wide}
\usepackage[margin=1in]{geometry}
\usepackage{graphicx}
\usepackage{fancyhdr}
\usepackage{mathtools}
\usepackage{topcapt}
\usepackage{booktabs}
\usepackage{comment}
\usepackage[scaled]{helvet}
\usepackage[T1]{fontenc}
\usepackage{mathrsfs}
\input xy
\xyoption{all}

\newcommand{\Z}{\mathbb{Z}}

\newcommand{\C}{\mathbb{C}}

\newcommand{\A}{\mathcal{A}}

\newcommand{\inv}{^{-1}}
\newcommand{\Hom}{\text{\rm Hom}}
\newcommand{\Mat}{\text{\rm Mat}}

\newcommand{\GL}{\text{\rm GL}}

\renewcommand{\O}{\mathcal{O}}
\newcommand{\cF}{\mathcal{F}}
\newcommand{\cC}{\mathcal{C}}
\newcommand{\cO}{\mathcal{O}}
\newcommand{\cM}{\mathcal{M}}

\newcommand{\U}{\mathcal{U}}
\newcommand{\cA}{\mathcal{A}}

\newcommand{\Spec}{\text{\rm Spec}}

\newcommand{\Proj}{\text{\rm Proj}}

\newcommand{\id}{\text{\rm Id}}

\newcommand{\dmod}{\text{\sf -mod}}
\newcommand{\Rep}{\mathsf{Rep}}
\newcommand{\End}{\text{\rm End}}

\newcommand{\Loc}{\mathsf{Loc}}

\newcommand{\Dq}{\mathcal{D}_q} 
\newcommand{\Dqc}{\mathcal{D}_q(\C)}
 
\newcommand{\Dqn}{\mathcal{D}_q(\C^n)}

\newcommand{\Zln}{Z_n}

\newcommand{\qq}{q^2}
\newcommand{\elm}{E}

\newcommand{\D}{\mathcal{D}}

\newcommand{\Dql}{{\mathcal D}_\lambda}

\newcommand{\qmm}{\tilde \mu}
\newcommand{\phidaggerstar}{\phi^{\dagger, \#}}

\newcommand{\Tell}{T^{(\ell)}}
\newcommand{\Kell}{K^{(\ell)}}

\newcommand{\Hell}{H^{(\ell)}}
\newcommand{\eell}{^{(\ell)}}
\newcommand{\Frl}{\text{\rm Fr}_\ell}

\newcommand{\Dqk}{\D_q(\mathcal K)}
\newcommand{\Xk}{\mathfrak{M}(\mathcal K)}
\newcommand{\frakM}{\mathfrak{M}}
\newcommand{\Xkl}{\Xk^{(\ell)}}

\newcommand{\Dqyk}{\D_{q}^Y(\mathcal K)} 
\newcommand{\oxr}{\O(X)[r\inv]}
\newcommand{\oyr}{\O(Y)[r\inv]}

\newcommand{\Fqmm}{Frobenius quantum moment map }

\newcommand{\g}{\mathfrak{g}}

\newcommand{\frakt}{\mathfrak{t}}
\newcommand{\frakk}{\mathfrak{k}}

\newcommand{\Lie}{\text{\rm Lie}}

\theoremstyle{plain}
\newtheorem{theorem}{Theorem}[section]
\newtheorem{prop}[theorem]{Proposition}

\newtheorem{lemma}[theorem]{Lemma}
\newtheorem{cor}[theorem]{Corollary}
\newtheorem{conj}[theorem]{Conjecture}

\theoremstyle{definition}
\newtheorem{definition}[theorem]{Definition}
\newtheorem{example}[theorem]{Example}
\newtheorem{rmk}[theorem]{Remark}
\newtheorem{notation}[theorem]{Notation}

\numberwithin{equation}{section}

\begin{document} 
\title{Quantizations of multiplicative hypertoric varieties at a root of unity}
\author{Iordan Ganev}% \fnref{now}}
%\date{August 2016}
\address{The University of Texas at Austin, Department of Mathematics, Austin, Texas, USA 78712}
\maketitle

\parskip = 0pt 

\begin{abstract} We construct quantizations of multiplicative hypertoric varieties using an algebra of $q$-difference operators on affine space, where $q$ is a root of unity in $\C$. The quantization defines a matrix bundle (i.e.\ Azumaya algebra) over the multiplicative hypertoric variety and admits an explicit finite \'etale splitting. The global sections of this Azumaya algebra is a hypertoric quantum group, and we prove a localization theorem. We introduce a general framework of Frobenius quantum moment maps and their Hamiltonian reductions; our results shed light on an instance of this framework.  \end{abstract}

%{\Small {\bf Keywords}: Hypertoric variety; $q$-difference operators; Azumaya algebra; localization.}

\tableofcontents
\section{Introduction} \label{sec:intro}

In this paper, we use $q$-difference operators to construct a class of Azumaya algebras on multiplicative hypertoric varieties, for $q$ a root of unity in $\C$. We identify their algebras of global sections with central reductions of hypertoric quantum groups. This construction leads to a description of representations of the hypertoric quantum group at a root of unity as coherent sheaves on the multiplicative hypertoric variety, via an analogue of the Beilinson-Bernstein localization theorem.

Our results can be understood as an instance of the paradigm that Azumaya algebras descend to Azumaya algebras under Hamiltonian reduction. This paradigm is formalized by the notion of Frobenius quantum moment maps and their Hamiltonian reductions, resulting in a framework that captures the appropriate level of generality to construct and study Azumaya algebras over a swath of classical moduli spaces constructed by group-valued Hamiltonian reduction. This framework encompasses: (1) the present setting of multiplicative hypertoric varieties, (2) Hilbert schemes and double affine Hecke algebras \cite{VV}, and (3) quantum $D$-modules on flag varieties \cite{BaKr} -- each of which lead to Azumaya algebras over their classical degenerations -- and also (4) multiplicative quiver varieties \cite{DJ}, where such results are expected.  Frobenius quantum Hamiltonian reduction is also a natural quantum/multiplicative analogue of the framework of \cite{BFG} in positive characteristic, deepening the rich parallels between the representation theory of quantum groups at roots of unity and universal enveloping algebras in positive characteristic. 

We briefly describe the notion of a Frobenius quantum moment map. Let $G$ be a reductive algebraic group over $\C$ with Lie algebra $\g$ and let $q$ be a primitive $\ell$-th root of unity. Let $U_q\g$ denote the (unrestricted) quantum group and let $\cO_q(G)\subset U_q\g$ denote the left coideal subalgebra of  ad-locally finite elements\footnote{Hence, $\cO_q(G)$ denotes the ad-equivariant quantized coordinate algebra, as opposed to the bi-equivariant `FRT' algebra.}. The algebra $\O_q(G)$ contains a central subalgebra $\O_q\eell(G)$ isomorphic to the algebra $\O(G)$ of functions on $G$.  Let $u_q\g$ denote\footnote{The algebra $u_q\g$ is a subalgebra of a quotient of Lusztig's small quantum group.} the quotient of $\O_q(G)$ by the ideal generated by the augmentation ideal of $\O_q\eell(G)$.
  
\begin{definition}\label{def:fqmm} A {\it \Fqmm} is the data of an algebra $A$ in the category of locally finite $U_q\mathfrak{g}$-modules, a central subalgebra $A^{(\ell)}\subset A$, and a quantum moment map $\mu_q :\cO_q(G)\to A$, such that $\mu_q (\cO^{(\ell)}_q(G)) \subset A^{(\ell)}$.  We denote by $\mu_q^{(\ell)}: \cO^{(\ell)}_q(G) \to A^{(\ell)}$ the restriction.
\end{definition}

\renewcommand{\qmm}{\mu_q}
\newcommand{\mufiber}{\mu_{\text{fiber}}}

Suppose that $A$ is Azumaya over $A\eell$. Thus, $A$ defines a vector bundle on $\Spec(A\eell)$ whose fibers are matrix algebras (assuming, for example, that $A\eell$ is finitely generated over $\C$).  Fix a point $\lambda \in \Spec(A\eell)$ lying over the augmentation ideal of $\O_q\eell(G)$, and an identification of the fiber of $A$ over $\lambda$ with the matrix algebra $\Mat_n(\C)$ for some $n$. Then there is an induced quantum moment map on fibers $\mufiber : u_q\g \rightarrow \Mat_n(\C)$ making the diagram below commute:
\begin{equation}\label{diag:fqmm}\xymatrix{ A\eell  \ar@{^{(}->}[rr] && A \ar[rr] && \Mat_n(\C)\\
\cO^{(\ell)}_q(G) \ar@{^{(}->}[rr] \ar[u]^{\qmm^{(\ell)}} && \cO_q(G) \ar[u]^{\qmm} \ar[rr] && u_q\g \ar[u]^{\mufiber}}.\end{equation}
The homomorphism $\mu_q\eell$ corresponds to a group-valued moment map $\Spec(A\eell) \rightarrow G$ in the sense of \cite{AMM}, possibly under additional hypotheses. Multiplicative GIT Hamiltonian reduction along $\qmm\eell$ produces a variety $\cM$, and it is expected that quantum Hamiltonian reduction along $\mu_q$, with respect to an appropriate integral form of the quantum group, defines an Azumaya algebra on $\cM$. 

\renewcommand{\qmm}{\mu}

In this paper, we establish results about Frobenius quantum moment maps in one of the simplest and most explicit contexts, namely, when $G = K$ is a subtorus of $(\C^\times)^n$ and $A = \Dqn$ is a certain algebra of $q$-difference operators on $\C^n$. In this case,  $\cM$ is a so-called multiplicative hypertoric variety, and we are naturally led to hypertoric geometry -- a hyperk\"ahler analogue of toric geometry.

Hypertoric varieties were introduced by Bielawski and Dancer \cite{BD}, and are one of the main examples of  hyperk\"ahler symplectic manifolds. As such, they have become a central object in modern representation theory (\cite{BeKu, BLPW:catO, HauselStrumfels, McBreenShenfeld, Stadnik}). Multiplicative hypertoric varieties are group-like analogues of (ordinary)  hypertoric varieties, and retain many of their nice algebraic and geometric properties. In particular, they are symplectic varieties, and may be explicitly constructed via a group-valued Hamiltonian reduction procedure. The quotients appearing in this construction may be taken either as GIT or affine quotients; the former choice produces a symplectic resolution of the latter.  The resulting symplectic resolution is not conical, so multiplicative hypertoric varieties provide an interesting class of examples to explore in order to extend the work of \cite{BLPW:symres, BDMN} beyond the conical setting. Whereas differential operators feature in the quantization of (ordinary) hypertoric varieties, multiplicative hypertoric varieties are quantized by considering $q$-difference operators, which in turn admit a close link to quantum groups. 

Let us now outline the main results of this paper. Fix a subtorus $K$ of $(\C^\times)^n$. Let $\frakM$ be a smooth multiplicative hypertoric variety corresponding to $K$, and let $\frakM\eell$ denote its Frobenius twist. There is a GIT quotient map $\pi : X^{\text{ss}} \rightarrow \frakM\eell$, where $X^{\text{ss}}$ is the set of semistable points of the fiber of a certain torus-valued moment map on $T^*\C^n$. A quotient of $\Dqn$ by a quantum moment map ideal defines a coherent sheaf on $X$, and we consider the pushforward of this sheaf along $\pi$. Taking $K$-invariants, we obtain a sheaf of algebras $\D$ over $\frakM\eell$. The sheaf $\D$ can be regarded as a quantization of the algebra of functions on the (non-twisted) multiplicative hypertoric variety $\frakM$. 

\begin{theorem}[Theorem \ref{thm:Dqkazumaya}] The sheaf $\D$ is an Azumaya algebra over $\frakM\eell$. \end{theorem}

Any Azumaya algebra is locally the endomorphism algebra of a vector bundle, but this property may or may not hold globally. When it does, the Azumaya algebra is said to {\it split}. We show that, while $\D$ itself does not split, its pullback along a finite \'etale cover of $\frakM\eell$ splits. Furthermore, $\D$ splits over the fibers of the affinization map of $\frakM\eell$ (which is a symplectic resolution), as well as over the fibers of a certain `residual' moment map on $\frakM\eell$. Let $\U = \Gamma(\frakM\eell, \D)$ be the algebra of global sections of  $\D$. The algebra $\U$ is (a localization of) a central reduction of a hypertoric quantum group and there are adjoint functors $\Loc: \U\dmod \leftrightarrows \D\dmod : \Gamma.$ We apply general principles, developed in \cite{BeKa, BMR}, to prove the following localization theorem:

\begin{theorem}[Theorem \ref{thm:localization}] For generic GIT parameter values, the derived functors of $\Loc$ and $\Gamma$  define inverse equivalences of triangulated categories $$D^b\left(\U\dmod\right) \stackrel{\sim}{\longrightarrow} D^b\left(\D\dmod\right).$$ \end{theorem}

Our results are inspired by, and parallel to, a number of Azumaya phenomena and localization theorems that occur in geometric representation theory over positive characteristic. For example,  Bezrukavnikov, Mirkovi\'c, and Rumynin \cite{BMR} prove a localization theorem relating the universal enveloping algebra of a semisimple Lie algebra in positive characteristic to the sheaf $\D(G/B)$ of (crystalline) differential operators on the corresponding flag variety. A key fact is that  $\D(G/B)$ is Azumaya over  the Frobenius-twisted cotangent bundle $T^*(G/B)^{(1)}$. In \cite{Stadnik}, Stadnik  gives a construction of an Azumaya algebra on hypertoric varieties in positive characteristic whose global sections is a central reduction of the corresponding hypertoric enveloping algebra. The starting point is the Azumaya property of the Weyl algebra of differential operators in positive characteristic, and a consequence is a localization theorem for hypertoric enveloping algebras.

Azumaya algebras also have a precedent in the context of quantizations at a root of unity. Backelin and Kremnitzer  \cite{BaKr} prove a localization theorem relating the representation theory of the quantum group $U_q\g$ at a root of unity to modules for a quantum counterpart $\Dq(G/B)$ of the algebra of differential operators on the flag variety. They show that $\Dq(G/B)$ is Azumaya over a dense open subset of the cotangent bundle $T^*(G/B)$. Varagnolo and Vasserot \cite{VV} use quantum differential operators on $\GL_n$ at a root of unity to construct an Azumaya algebra on a deformed Hilbert scheme. The global sections of this algebra is the spherical double affine Hecke algebra at a root of unity, and a localization theorem holds. 

The present work combines aspects of \cite{Stadnik} and \cite{BaKr} by using a $q$-deformation of the Weyl algebra at a root of unity to define an Azumaya algebra on a multiplicative hypertoric variety (in characteristic zero). The proof of our localization theorem adopts the same techniques as those appearing in \cite{BMR}.  We remark that Bellamy and Kuwabara construct quantizations of hypertoric varieties using W-algebras \cite{BeKu}. Their construction does not hold in the multiplicative setting, but can be viewed as a degeneration of the one presented here. We also note that Levitt and Yakimov consider certain algebraic properties of quantized Weyl algebras at a root of unity \cite{LY}.

In  current work-in-progress, various aspects of which are joint with Nicholas Cooney and David Jordan, we intend to generalize the construction of the sheaf $\D$ to non-abelian reductive groups $G$.  A step toward this goal is the investigation of multiplicative quiver varieties. When the dimension vector of a quiver is one at all vertices, the resulting multiplicative quiver variety is an example of a multiplicative hypertoric variety. It is expected that the results of this paper generalize to certain multiplicative quiver varieties with higher dimension vectors and their quantum versions (introduced in \cite{DJ}) at a root of unity. 

The layout of this paper is as follows. Section \ref{sec:prelim} contains preliminaries, including relevant facts about Azumaya algebras, torus-valued moment maps, and quantum moment maps. Section \ref{sec:prelim} also includes the definition of a braided deformation of the category of representations of a torus and the definition of multiplicative hypertoric varieties and their Frobenius twists. In Section \ref{subsec:Dqn:def}, we consider a version $\Dqn$ of the quantum Weyl algebras introduced in \cite{GZ} that depends on a number $q \in \C^\times$ and a subtorus of $(\C^\times)^n$. Section \ref{subsec:Dqn:quiver} includes examples of the algebras $\Dqn$ coming from quivers. When $q$ is an $\ell$-th root of unity, the center of  $\Dqn$ is isomorphic to the algebra $\O(T^*\C^n)$ of functions on the cotangent bundle $T^*\C^n$ (Section \ref{subsec:Dqn:center}). We compute its Azumaya locus in Sections \ref{subsec:Dqn:azumaya}, and construct an \'etale splitting in Section \ref{subsec:Dqn:etale}. In Section \ref{subsec:Dqk:def}, we use the algebra $\Dqn$ to construct a sheaf of algebras on the Frobenius twist of any smooth multiplicative hypertoric variety. This algebra is Azumaya, as proven in Section \ref{subsec:Dqk:azumaya}, and  an explicit finite \'etale splitting is given in Section \ref{subsec:Dqk:etale}. An example related to the affine $A_n$ quiver is given in Section \ref{subsec:Dqk:quiver}.  We prove a localization theorem in Section \ref{subsec:Dqk:localization} that relates modules for the Azumaya algebra with modules for the its algebra of global sections, which is a central reduction of a hypertoric quantum group.

{\bf Acknowledgements.} The author is grateful to David Jordan for suggesting this project and providing guidance throughout, particularly for the formulation of Frobenius quantum moment maps and key ideas in the proofs of Theorems \ref{thm:Dqnazumaya} and \ref{thm:Dqkazumaya}. Special thanks to David Ben-Zvi (the author's PhD advisor) for numerous discussions and constant encouragement, and for suggesting the term `hypertoric quantum group.' Many results appearing in the current paper were proven independently by Nicholas Cooney; the author is grateful to Nicholas for sharing his insight on various topics, including Proposition \ref{prop:center}. The author also thanks Nicholas Proudfoot for relating the definition of multiplicative hypertoric varieties, as well as the content of Remark \ref{rmk:hypertoric}. The author also benefited immensely from the close reading and detailed comments of an anonymous referee, and from conversations with Justin Hilburn, Kobi Kremnitzer, Michael McBreen, Tom Nevins, Travis Schedler, and Ben Webster. This work was supported by the National Science Foundation through a Graduate Research Fellowship, and through grant No.\ 0932078000 while the author was in residence at the Mathematical Sciences Research Institute during Fall 2014.

%%%%%%%%%%%%%%%%%%%%%%%%%%%%
\section{Preliminaries}\label{sec:prelim}

\subsection{Azumaya algebras }\label{subsec:pre:azumaya}

In this  section, we recall relevant results from theory of Azumaya algebras. References include \cite{BrownGoodearl, Milne:1980bb, Saltman}. Let $R$ be a commutative ring. For $p \in \Spec(R)$, let $k(p)$ denote the residue field at $p$. If $M$ is a right $R$-module, the {\it fiber} of $M$ over $p$ is defined as the $k(p)$-vector space $M \otimes_R k(p)$, and the {\it rank} of $M$ at $p$ is the dimension of $M \otimes_R k(p)$. 

\begin{definition} Let $R$ be a commutative ring. An $R$-algebra $A$ is called  {\it Azumaya} if $A$ is finitely generated and  projective\footnote{Some authors require $A$ to have non-vanishing rank. This condition guarantees that Azumaya algebras are precisely the $R$-algebras $A$ such that $A \otimes_R B$ is Morita equivalent to $R$ for some $R$-algebra $B$.} as an $R$-module, and the natural map $$A \otimes_R A^\text{opp} \rightarrow \End_R(A), \qquad  \qquad a \otimes b \mapsto (c \mapsto acb)$$  is an isomorphism. An Azumaya $ A$ is called {\it split} if there is a projective $R$-module $P$ and an isomorphism $ A \simeq  \End_R (P).$ \end{definition}

An Azumaya algebra over a field $k$ is the same as a central simple $k$-algebra. If $k$ is algebraically closed, this is the same as a matrix algebra over $k$. Azumaya algebras admit the following local characterization. 

\begin{lemma} Suppose $A$ is an $R$-algebra that is finitely generated and projective as an $R$-module. Then $A$ is an Azumaya $R$-algebra if and only if the fiber $A \otimes_R k(p)$ is a central simple $k(p)$-algebra for all closed points $p \in \Spec(R)$.\end{lemma}

\begin{definition} Let $A$ be an algebra that is finitely generated and projective as a module over its center $Z$. The {\it Azumaya locus} of $ A$ is defined as  $ \{ p \in \Spec(Z) \ | \ \text{$A\otimes_Z k(p)$ is Azumaya over $k(p)$} \}.$ \end{definition}

\begin{definition} Let $X$ be a scheme. A sheaf $\mathcal A$ of $\O(X)$-algebras is called an {\it Azumaya algebra on $X$} if ${\mathcal A}$ is locally free of finite rank and the natural map ${\mathcal A} \otimes_{\O(X)} {\mathcal A}^{\text{opp}} \rightarrow {\mathcal E}nd_{\O(X)}({\mathcal A})$ is an isomorphism. An Azumaya algebra $\D$ is called {\it split} if there is a locally free sheaf $\mathcal F$ on $X$ and an isomorphism ${\D} \simeq {\mathcal E}nd_{\O(X)}(\cF).$ An {\it \'etale splitting} of an Azumaya algebra $\mathcal A$ on $X$ is an \'etale cover $f: Y \rightarrow X$ such that the pullback $f^*\mathcal A$ is a split Azumaya algebra on $Y$.\end{definition}

\begin{prop}\label{prop:AzumayaCartesian} Let $\cA$ be an Azumaya algebra on $X$ and let $f: Y \rightarrow X$ be an \'etale splitting of $\cA$ that fits into a Cartesian square $$\xymatrix{Y \ar[r]^f \ar[d]^\nu & X \ar[d]^\mu\\ C \ar[r]^g & B}. $$ Then $\cA$ splits along fibers of the map $\mu : X \rightarrow B$. \end{prop}

\begin{proof} Let $b\in B$ with $\mu\inv(b) \neq \emptyset$. Then, since $f$ is surjective, there exists $c\in C$ such that $g(c) = b$. The Cartesian property of the square implies that the restriction $f' := f\vert_{\nu\inv(c)} : \nu\inv(c) \rightarrow \mu\inv(b)$ is an isomorphism of schemes.  By hypothesis, $f^*\cA = \End_Y(\cF)$ for a locally free sheaf $\cF$ on $Y$. It follows that $\cA\vert_{\mu\inv(b)} $ $= $ $f_*' (f')^* (\cA\vert_{\mu\inv(b)}) = $ $ \End_{\mu\inv(b)}(f_*'(\cF\vert_{\nu\inv(c)})) .$\end{proof}

Recall that, if $R$ is a finitely generated algebra over a field, and if $\Spec(R)$ is connected, then an $R$-module is projective if and only if its rank is constant on closed points of $\Spec(R)$. A similar statement holds for connected schemes of finite type over a field. The following proposition is adapted from Theorem 2.2 of \cite{Saltman}.

\begin{prop} Suppose $X$ is a connected scheme of finite type over an algebraically closed field. Then a sheaf ${\mathcal A}$ of $\O(X)$-algebras is an Azumaya algebra if and only if (1) the rank of $\mathcal A$ as a sheaf of $\O(X)$-modules is constant on closed points, and (2) the fiber ${\mathcal A} \otimes_{\O(X)} k(x)$ is isomorphic to a matrix algebra for all closed points $x \in X$. \end{prop}

\subsection{Tori and moment maps} \label{subsec:pre:tori}

Throughout this paper, we will adopt the following notation. Let $K = (\C^\times)^d$  and $T = (\C^\times)^n$ be standard tori of rank $d$  and $n$ with $d \leq n$.  Write
$$\O(K) = \C[z_j^{\pm1}] = \C[z_j^{\pm1} \ | \ j = 1, \dots, d], \qquad \O(T) = \C[y_i^{\pm1}] = \C[ y_i\ | \ i = 1, \ldots, n]$$ for the algebras of functions on $K$ and $T$. Let $\frakk$ and $\frakt$ be the Lie algebras of $K$ and $T$. Identify the character lattices $X^*(K)$ and $X^*(T)$  with $\Z^d$ and $\Z^n$, and write $z^\mathbf r \in \O(K)$ for the monomial corresponding to $\mathbf r \in \Z^d = X^*(K)$. Fix an inclusion $\phi : K \hookrightarrow T$. Thus, $\phi$ has the form $$\phi(k)_i = \prod_{j=1}^d k_j^{m_{ij}}, \qquad \text{ $i = 1, \dots,n$},$$ 
for some integers $m_{ij} \in \Z$. Let $M$ denote the $n$ by $d$ matrix defined by the integers $m_{ij}$. There is a `transpose' map $\phi^\dagger : T \rightarrow K$ defined by $\phi^\dagger(t)_j = \prod_{i=1}^n t_i^{m_{ij}}$ and an algebra homomorphism $\phidaggerstar : \O(K) \rightarrow \O(T)$ obtained by pulling back functions along $\phi^\dagger$.  Let $H = T/K$ be the quotient torus. Denote by $\psi : T \rightarrow H$ the quotient map, and $\psi^\dagger: H \rightarrow K$ its transpose. 

Suppose $T = (\C^\times)^n$ acts on a smooth symplectic variety $(X,\omega)$ preserving the symplectic form. For $\xi \in \frakt$, let $v_\xi$ denote the corresponding vector field on $X$. Let $\langle , \rangle$ denote the usual dot product $\frakt = \C^n$, which we use to identify $\frakt \simeq \frakt^*$. Let $\theta  \in \Omega^1(T, \frakt)$ denote the left-invariant Maurer-Cartan form on $T$. Throughout, we consider the adjoint (equivalently, trivial) action of $T$ on itself. 

\begin{definition}\label{def:torusvaluedmm} A  {\it (torus-valued) moment map} is a smooth, $T$-equivariant map $\mu: X \rightarrow T$ such that $$\omega(v_\xi, -) = \langle \mu^*\theta, \xi \rangle$$ for all $\xi \in \frakt$. The corresponding homomorphism $\mu^\#: \O(T) \rightarrow \O(X)$ is called a {\it comoment map}.  \end{definition}

In coordinates, the Maurer-Cartan form can be written as $\theta = \sum_i t_i\inv d_i$, and the moment map condition becomes: \begin{equation} \label{eq:mmcoord}\omega(v_\xi, -) = \sum_{i=1}^n \frac{d\mu_i}{\mu_i} \xi_i,\end{equation} where $\mu_i$ is the composition of $\mu$ with the $i$th projection. 

\begin{prop}\label{prop:functoriality} If $ \mu :  X \rightarrow T$ is a moment map, then the composition $\phi^\dagger \circ  \mu$ is a moment map for the action of $K$ on $X$ induced by $\phi$. \end{prop}

\begin{proof} The linear maps $\text{Lie}(\phi) : \frakk \rightarrow \frakt$ and $\Lie(\phi^\dagger) : \frakt \rightarrow \frakk$ corresponding to $\phi$ and $\phi^\dagger$ are transposes of one another. We use the same notation for the maps on 1-forms: $\Lie(\phi) : \Omega^1(-, \frakk) \rightarrow \Omega^1(-, \frakt)$ and $ \Lie(\phi^\dagger) : \Omega^1(-, \frakt) \rightarrow \Omega^1(-, \frakk).$ These maps commute with the pullback of 1-forms along smooth maps. Let $\theta_T$ and $\theta_K$ be the Maurer-Cartan forms on $T$ and $K$. Then\footnote{In fact, for any group homomorphism $\beta : G_1 \rightarrow G_2$, it is easy to show that $\beta^*\theta_{G_2} = \Lie(\beta)(\theta_{G_1})$.} $(\phi^\dagger)^*\theta_K = \Lie(\phi^\dagger) ( \theta_T)$. For  $\zeta \in \frakk$, write $v_\zeta^K$ for the vector field on $X$ generated by $\zeta$. In fact, $v_\zeta^K$ coincides with the vector field $v_{\text{Lie}(\zeta)}^T$ generated by the image of $\zeta$ in $\frakt$. The remainder of the proof is a computation: \begin{align*}\langle (\phi^\dagger\circ\mu)^* \theta_K, \zeta \rangle &= \langle \mu^* \circ (\phi^\dagger)^* \theta_K, \zeta \rangle = \langle \mu^*(\text{Lie}(\phi^\dagger) ( \theta_T)), \zeta \rangle = \langle \Lie(\phi^\dagger)( \mu^* \theta_T), \zeta \rangle \\ &=  \langle  \mu^* \theta_T,   \text{Lie}(\phi)( \zeta) \rangle = \omega(v^T_{\text{Lie}(\zeta)}, -) = \omega(v^K_\zeta, - ). \end{align*} \end{proof}

For $t \in T$, let $L_t : T \rightarrow T$ denote the action of $t$ by left multiplication. 

\begin{prop}\label{prop:mmfactors} Suppose $\mu : X \rightarrow T$ is a moment map, and the action of $K$ on $X$ is trivial.  Then there is an induced action of $H$ on $X$ and a moment map $ \mu_H : X \rightarrow H$ that satisfies $\mu = \mu_H \circ \phi^\dagger\circ L_{t_0}$ for some $t_0 \in T$. \end{prop}

\begin{comment}
In other words, the following diagram commutes:\[ \xymatrix{ & & X \ar[d]_\mu \ar@{-->}[lld]_{\mu_H} \\ H \ar[r]_{\phi^\dagger} & T \ar[r]_{L_{t_0}} & T }\] 
\end{comment}

\begin{proof}  Let $\mathfrak{h} $ denote the Lie algebra of $H$. As in the proof of Proposition \ref{prop:functoriality}, we have Lie algebra homomorphisms $\text{Lie}(\psi): \frakt \rightarrow \mathfrak{h}$ and $\text{Lie}(\psi^\dagger) = \text{Lie}(\psi)^T : \mathfrak{h} \rightarrow \frakt$. The short exact sequence of Lie algebras $0 \rightarrow \mathfrak{h} \rightarrow \frakt \rightarrow \frakk \rightarrow 0$, with maps $\Lie(\psi^\dagger)$ and $\Lie(\phi^\dagger)$, exponentiates to a short exact sequence $$1 \longrightarrow H \stackrel{\psi^\dagger}{\longrightarrow} T \stackrel{\phi^\dagger}{\longrightarrow} K \longrightarrow 1.$$ 

Fix $x_0 \in X$ and let $t_0 = \mu(x_0) \in T$. Since the action of $K$ on $X$ is trivial and $\phi^\dagger \circ \mu : X \rightarrow K$ is a moment map, it is straightforward to show that $\phi^\dagger \circ \mu$ is constant. Thus,  $\mu(x)t_0\inv \in \text{Ker}(\phi^\dagger) = \text{Im}(\psi^\dagger)$ for any $x \in X$. Using the fact that $\phi^\dagger$ is injective, define $\mu_H : X \rightarrow H$ by  $x\mapsto (\phi^\dagger)\inv\left(\mu(x)t_0 \inv\right).$

We show that $\mu_H$ is a moment map. Let $\xi \in \frakt$. The vector field $v_{\Lie(\psi)(\xi)}^H$ corresponding to the image of $\xi$ in $\mathfrak{h}$  coincides with the vector field $v_{\xi}^T$ corresponding to $\xi$. Since $\Lie(\psi)$ is surjective, the result is a consequence of the following computation, which uses facts stated in the first  proof of Proposition \ref{prop:functoriality}, and the left-invariance of $\theta_T$:
\begin{align*} \omega(v_{\Lie(\psi)(\xi)}^H, - )& = \omega(v^T_\xi, - ) =  \langle  \mu^* \theta_T,   \xi \rangle = \langle  \mu_H^* \circ (\psi^\dagger)^* \circ L_{t_0}^* \theta_T,   \xi \rangle = \langle  \mu_H^* \circ (\psi^\dagger)^* \theta_T,   \xi \rangle \\ &= \langle  \mu_H^* (\Lie(\psi^\dagger)( \theta_H)),   \xi \rangle = \langle  \Lie(\phi^\dagger)( \mu_H^*  \theta_H),   \xi \rangle = \langle  \mu_H^*  \theta_H, \Lie(\phi)(  \xi) \rangle. \end{align*}  \end{proof}

\begin{rmk} We make the following remarks:
\begin{enumerate}
 \item The map $\phidaggerstar$ is  natural in the context of quantum groups. The Drinfeld-Jimbo quantum groups $U_q(\mathfrak k)$ and $U_q(\mathfrak t)$ can be identified with $\O(K)$ and $\O(T)$, respectively. Under this identification, the  homomorphism $U_q(\mathfrak k) \rightarrow U_q(\mathfrak t)$ induced by $\phi$ coincides with $\phidaggerstar$. 
 
 \item The notion of a group-valued moment map is defined in \cite{AMM}, and simplifies to Definition \ref{def:torusvaluedmm} when the group is a torus. This simplification is possible since the left- and right-invariant Maurer-Cartan forms on a torus coincide. One may argue that the moment map is more  naturally valued in the dual torus, which avoids the identification $\frakt \simeq \frakt^*$. However, we have chosen to be consistent with \cite{AMM}.
 
 \item  The proof of Proposition \ref{prop:functoriality} can be adapted to the case of an arbitrary homomorphism $\phi :K \rightarrow T$ of tori.
\end{enumerate}
\end{rmk}

\subsection{Multiplicative hypertoric varieties}\label{subsec:pre:hypertoric} We adopt the notation of the previous section. There is an action of $T = (\C^\times)^n$ on the cotangent bundle $T^*\C^n$ by componentwise scaling: $(t \cdot (p, w))_i$ $=$  $(t_i p_i, t_i\inv w_i)$, for $i = 1, \dots, n$. Precomposition by $\phi$ induces an action of $K$ on $T^*\C^n$. Fix the following notation: $$\O(T^*\C^n) = \C[x_i, \partial_i] , \qquad \O(T^*\C^n)^\circ = \C[x_i, \partial_i][(1 + x_i \partial_i)\inv].$$ $$(T^*\C^n)^\circ = \{ (p,w) \in T^*\C^n \ : \  1+ p_i w_i \neq 0 \}.$$ Equip $(T^*\C^n)^\circ$ with the symplectic form $\omega = \sum_i \frac{dp_i \wedge dw_i}{1+p_iw_i}.$ 

\begin{prop}\label{prop:TCnmomentmap} The following are  torus-valued moment maps:  \begin{align*} \mu_T : (T^*\C^n)^\circ & \rightarrow T  & \qquad \mu_K : (T^*\C^n)^\circ &\rightarrow K   \\ (p,w) &\mapsto  (1+ p_i w_i)_i & \qquad   (p,w) &\mapsto (\prod_{i=1}^n (1+ p_i w_i)^{m_{ij}})_j. \end{align*} \end{prop}

\begin{proof} By Proposition \ref{prop:functoriality}, the claim for $\mu_K$ is a consequence of the claim for $ \mu_T$. We consider the case $n=1$ and $T = \C^\times$; the general case is similar. For $\xi \in \frakt = \C$, one checks directly that $$\omega (v_\xi, - ) =  \frac{dp \wedge dw}{1+pw} \left( \xi p \frac{\partial}{\partial p} -  \xi w \frac{\partial}{\partial w}, -\right) =   \xi \frac{ w dp +  p dw}{1+p w} = \xi \frac{d\mu_T}{\mu_T} = \langle \mu_T^*\theta_T, \xi \rangle.$$  \end{proof}

The maps above give rise to comoment maps $\mu^\#_T: \O(T)\rightarrow \O(T^*\C^n)^\circ$ and $\mu^\#_K:$  $\O(K)$ $= $  $\C[z_j^{\pm 1}]  \rightarrow \O(T^*\C^n)^\circ.$ The following definition was communicated to us by N.~Proudfoot and T.~Hausel.  
 
\begin{definition} Let ${\mathcal K} = (K, \eta, \chi)$ be the data of (1) a connected subtorus $K$ of $T = (\C^\times)^n$, (2) a point $\eta$ of $K$, and (3) a character $\chi \in X^*(K)$. The {\it multiplicative hypertoric variety} corresponding to the data $\mathcal K$ is defined as the GIT quotient $$\Xk := \mu_K\inv(\eta) \text{ $\!$/$\! \!$/$\!$}_\chi  K.$$ We call the data ${\mathcal K} = (K, \eta, \chi)$ {\it smooth} if $\Xk$ is a smooth variety. \end{definition}

\begin{comment}
Explicitly, $$\Xk = \Proj \left( \bigoplus_{n \geq 0} \left( \O(T^*\C^n)^\circ/ ( \mu_K^\# - \eta) \right)_{\chi^n} \right),$$ where $(\mu_K^\# - \eta)$ denotes the ideal generated by the elements  $\mu_K^\#(z_j) - \eta_j$ for $j = 1, \ldots, d$. 
\end{comment}

\begin{prop} Suppose $\mathcal K$ is smooth. Then $\Xk$ is a symplectic variety and there is a moment map $\Xk \rightarrow H$ for the action of the torus $H = T/K$ on $\Xk$. \end{prop}

\begin{proof}There is a cover of $\Xk$ by affine open sets of the form $U_r = \Spec(\O(\mu\inv(\eta))[r\inv]^K)$ where $r \in \O(\mu\inv(\eta))_{\chi^n}$ is $\chi^n$-invariant for $n >0$. Each $U_r$ acquires a symplectic form and has a Hamiltonian action of $T$ with $K$ acting trivially. Thus, $H$ acts on $U_r$ and, by Proposition \ref{prop:mmfactors}, there is a moment map $U_r \rightarrow H$. These actions and maps glue to give an action of $H$ on $\Xk$ and a moment map $\Xk \rightarrow H$. \end{proof}

\newcommand{\MzK}{\mathfrak{M}_0(\mathcal K)}

Let $\MzK = \frakM(K,\eta,0)= \Spec(\O(\mu_K\inv(\eta))^K)$. Then $\MzK$ is the affinization of $\Xk$, and we have a projective morphism $\nu : \Xk \rightarrow \MzK.$ When $\mathcal K$ is smooth, the map $\nu$ is a resolution of singularities. 

\begin{rmk}\label{rmk:hypertoric} We call the pair $(\eta,\chi)$ {\em generic} if $K$ acts locally freely on the semistable locus of $\mu_K^{-1}(\eta)$, or (equivalently) if the semistable locus coincides with the stable locus.  In this case, $\mathfrak{M}(\mathcal K)$ is a geometric quotient of the semistable locus by $K$, and has at worst orbifold singularities.  Such pairs always exist; for example, if $\eta$ is a regular value of $\mu_K$, then $(\eta,\chi)$ is generic for every $\chi$.  It would be interesting to give a combinatorial characterization of generic pairs analogous to the characterization for linear hypertoric varieties \cite[3.3]{BD}. If $(\eta,\chi)$ is generic, then $\mathfrak{M}(\mathcal K)$ will be smooth if and only if $K$ acts freely on the stable locus, which will be the case if and only if the embedding of $K$ in $T$ is unimodular.\end{rmk}

\begin{comment}
There is a more explicit description of the moment map  $U_r \rightarrow H$. One can identify $H$ with the kernel of $\phi^\dagger : T \rightarrow K$, and the algebra of functions $\O(H)$ can be realized as $\O(H) = \O(T)/ (\prod_{i=1}^n y_i^{m_{ij}} - \eta_j).$ The comoment map $\O(T) \rightarrow \O(\mu\inv(\eta))$, $y_i \mapsto 1+p_iw_i$ factors to give a homomorphism $\O(H) \rightarrow \O(\mu\inv(\eta))[r\inv]^K$ corresponding to a moment map  $U_r \rightarrow H$. 
\end{comment}

\begin{comment}
\begin{rmk} As far as the author is aware, it is not known whether multiplicative hypertoric varieties (in the sense of this paper) admit hyperk\"ahler metrics. In the case of multiplicative hypertoric varieties arising from quivers, a more general definition has been introduced by Boalch (see \cite{Boalch}). Some of these more general `multiplicative hypertoric varieties' arising from quivers are known to admit hyperk\"aler metrics. \end{rmk}
\end{comment}

\subsection{Frobenius twists}\label{subsec:pre:twisthypertoric}

Let $\ell$ be a positive integer. Recall that we write $\O(K) = \C[z_j^{\pm1}]$ for the algebra of functions on the standard torus $K = (\C^\times)^d$. 

\begin{definition} The {\it Frobenius twist} of the torus $K$ is defined as $\Kell := \Spec(\C[z_j^{\pm \ell}])$. Thus, there is a natural map $\Frl : K \rightarrow \Kell$ taking $k$ to $k^\ell$. \end{definition}

As a group, $\Kell$ is isomorphic to $K$. The map $\phi$ induces a map $\phi\eell : \Kell \rightarrow \Tell$. Similarly, we define the Frobenius twist of the cotangent bundle $T^*\C^n$ and a certain localization:
 $$\O((T^*\C^n)\eell) = \C[x_i^\ell, \partial_i^\ell] \subset \O(T^*\C^n) = \C[x_i, \partial_i], \qquad (T^*\C^n)\eell = \Spec(\O((T^*\C^n)\eell))$$ $$ \O((T^*\C^n)\eell)^{\circ} = \C[x_i^\ell, \partial_i^\ell][(1 + x_i^\ell \partial_i^\ell)\inv], \qquad (T^*\C^n)^{(\ell), \circ} = \Spec(\O((T^*\C^n)\eell)^\circ).$$ Inserting $\ell$-th powers where appropriate in the definitions of $\mu_K$ and $\mu_K^\#$, we obtain moment maps $$  \mu_{\Kell} : (T^*\C^n)^{(\ell), \circ} \rightarrow \Kell, \qquad \mu_{\Kell}^\# : \O(\Kell) \rightarrow  \O((T^*\C^n)\eell)^{\circ}.$$ 

\begin{definition} Given the data $\mathcal{K} = (K, \eta, \chi)$, define the corresponding {\it Frobenius twisted multiplicative hypertoric variety}  by $$\Xk^{(\ell)} = (\mu_{\Kell})\inv(\eta^\ell) \text{ $\!$/$\! \!$/$\!$}_{\chi\eell}  \Kell,$$ where $\chi\eell$ is the character of $\Kell$ defined by $\chi\eell(\Frl(k)) = \chi(k^\ell)$. \end{definition}

\begin{prop}\label{prop:frobenius} As varieties, $\Xk$ and $\Xk\eell$ are isomorphic. The torus $\Hell$ acts on $\Xkl$ there is a moment map $ \Xkl \rightarrow \Hell$. The inclusion $\C[x_i^\ell, \partial_i^\ell] \hookrightarrow \C[x_i, \partial_i]$ induces a finite map $\Frl : \Xk \rightarrow \Xk^{(\ell)},$ referred to as the Frobenius map on $\Xk$. \end{prop}

\begin{comment} \begin{proof} The inclusion $\C[x_i^\ell, \partial_i^\ell] \hookrightarrow \C[x_i, \partial_i]$ induces a finite map $$\left( \O((T^*\C^n)\eell)^{\circ} / ( \mu^{(\ell)}  - \eta^\ell) \right)_{\chi\eell} \rightarrow \left(\O(T^*\C^n)^{\circ} / ( \mu  - \eta) \right)_{\chi^\ell}.$$ Hence we a composition of finite maps $$\bigoplus_{n \geq 0}\left( \O((T^*\C^n)\eell)^{\circ} / ( \mu^{(\ell)}  - \eta^\ell) \right)_{(\chi\eell)^n} \rightarrow \bigoplus_{n \geq 0} \left(\O(T^*\C^n)^{\circ} / ( \mu  - \eta) \right)_{\chi^{n\ell}} \rightarrow \bigoplus_{m \geq 0} \left(\O(T^*\C^n)^{\circ} / ( \mu  - \eta) \right)_{\chi^{m}} .$$ \end{proof}
\end{comment}

\subsection{Braided deformations of $\Rep(K)$}\label{subsec:pre:braideddeform}

Let $\Rep(K)$ denote the tensor category of algebraic representations of $K = (\C^\times)^d$ over $\C$. Any object $V$ of $\Rep(K)$ admits a decomposition into isotypic components given by $V = \bigoplus_{\mathbf r \in \Z^d} V_{\mathbf r},$ where $$V_\mathbf r =  \{ v \in V \ | \ k\cdot v = \prod_{j} k_j^{r_j}v \ \text{for all} \ k= (k_j) \in K\}  .$$ This decomposition establishes an equivalence of categories between $\Rep(K)$ and the category of $\Z^d$-graded vector spaces. 

\begin{definition} \label{eq:nontrivbraid} Given $q \in \C^\times$ and a bilinear form $\langle, \rangle$ on $X^*(K)$, a braiding on $\Rep(K)$ is given by: \begin{equation} \sigma^{(q)}_{V,W} : V \otimes W \rightarrow W \otimes V $$ $$ v\otimes w \mapsto q^{\langle \mathbf r , \mathbf s \rangle} (w \otimes v),\end{equation} where $v \in V_{\mathbf r}$ and $w \in W_{\mathbf s}$ are homogeneous elements of degrees $\mathbf r$ and $\mathbf s$, respectively. The resulting braided tensor category is denoted $\Rep_q(K)$ or $\Rep_{(q, \langle, \rangle)}(K)$. \end{definition}

\begin{rmk} If $\langle, \rangle$ is  the dot product $\langle \mathbf r, \mathbf s \rangle = \sum_i r_i s_i$, then $\Rep_q(K)$ coincides with the Deligne tensor product of categories $\Rep_q(\C^\times)^{\boxtimes d}$.   \end{rmk}

Suppose $q \in \C^\times$ is a primitive $\ell$-th root of unity. Let $q_j$ be the image of $q$ under the $j$th inclusion $\C^\times \hookrightarrow K$, for $1 \leq j \leq d$. Let $\Gamma$ the group of $\ell$-th roots of unity in $K$, i.e.\ the subgroup generated by the  $q_j$. Any object $V$ of the category $\Rep(\Gamma)$  decomposes as $V = \bigoplus_{\mathbf r \in (\Z/\ell\Z)^d} V_{\mathbf r}$ with $V_{\mathbf r} = \{ v \in V \ | \ q_j \cdot v = q^{r_j} v \ \text{for $1 \leq j \leq d$} \}.$ 

\begin{definition} Define $\Rep_q(\Gamma)$ to be the braided tensor category whose underlying tensor category is $\Rep(\Gamma)$ with braiding given by the formula in Definition \ref{eq:nontrivbraid} with $\mathbf r, \mathbf s \in (\Z/\ell\Z)^d$.\end{definition}

\begin{notation} Henceforth, when we write $\Rep_q(K)$, we  assume that $q$ and $\langle, \rangle$ have been fixed; when we write $\Rep_q(\Gamma)$ we assume, in addition, that $q$ is a primitive $\ell$-th root of unity. \end{notation}

\subsection{Braided tensor products of matrix algebras}\label{subsec:pre:tensoralgebras}\label{subsec:pre:deformmatrix}

If $A$ and $B$ are algebras in a braided tensor category $({\mathcal C}, \otimes, \sigma)$, then the tensor product $A \otimes B$ inherits an algebra structure in $\cC$ given by  the composition $$m_{A \otimes B} : (A \otimes B) \otimes (A \otimes B) \stackrel{1 \otimes \sigma_{A,B} \otimes 1}{\longrightarrow} A \otimes A \otimes B \otimes B \stackrel{m_A \otimes m_B}{\longrightarrow} A \otimes B,$$ where $\sigma_{A,B} : A \otimes B \rightarrow B \otimes A$ denotes the braiding. This construction extends to the tensor product $A_1 \otimes \cdots \otimes A_r$ of an arbitrary (ordered) collection of algebras in $\cC$.

\begin{example} Suppose $A$ and $B$ are algebras in the tensor category $\Rep_q(K)$, i.e.\ graded algebras. The tensor product $A \otimes B$ has the same underlying vector space as the ordinary tensor product and multiplication defined by $$(a_1 \otimes b_1) \cdot (a_2 \otimes b_2) = q^{\langle \mathbf r, \mathbf s \rangle} (a_1 a_2) \otimes (b_1 b_2),$$ where $a_2 \in A_{\mathbf s}$ and $b_1 \in B_{\mathbf r}$ are homogeneous elements.\end{example}

In the following lemma (whose proof is an exercise), $\cC$ denotes either $\Rep_q(K)$ or $\Rep_q(\Gamma)$.

\begin{lemma} \label{lem:quot-tensor} Let $A$ and $B$ be algebras in $\cC$. Suppose $z \in A_{\mathbf 0}$ and $w \in B_{\mathbf 0}$ are invariant elements such that the left ideals $Az$ and $Bw$ are 2-sided. Then the left ideal $A \otimes B(z \otimes 1, 1 \otimes w)$ is 2-sided and there is an isomorphism of algebras in $\cC$: $$A/Az \otimes B/Bw \simeq (A \otimes B)/ A\otimes B(z \otimes 1, 1 \otimes w).$$\end{lemma}

Let $N$ be a positive integer and let $\Mat(N)$ denote the algebra of $N$ by $N$ matrices over $\C$. For $1 \leq i,j\leq N$, the elementary matrix $\elm_{i,j}$ is the matrix in $\Mat(N)$ with the entry 1 in the $i$th row and $j$th column and all other entries equal to 0. Any function $f: \{1,\dots, N\} \rightarrow \Z^d$ determines a $\Z^d$-grading on the algebra $\Mat(N)$ by setting $ \deg(\elm_{ij}) = f(i) - f(j)$ for an elementary matrix $\elm_{ij}$, where $1 \leq i,j\leq N$. Let $(\Mat(N), f)$ denote the matrix algebra with the $\Z^d$-grading determined by $f$. We regard $(\Mat(N),f)$ as an algebra in $\Rep_q(K)$ for any $q \in \C^\times$.

Let $M$ be a positive integer and consider a collection $\{ (\Mat(N_i), f_i)\}_{i=1}^M$ of $\Z^d$-graded matrix algebras. Set $N = \prod_{i=1}^M N_e$ and observe that every integer $n$ with $1 \leq n \leq N$ can be written as  $n = \sum_{i=1}^M N_1 \cdots N_{i-1} (n_i -1) + 1$ for unique $n_i$ satisfying $1 \leq n_i \leq  N_i$. Define $f: \{1,\dots, N\} \rightarrow \Z^d$  as $f( n) = \sum_{i=1}^M f_i(n_i).$

\begin{prop} \label{prop:deformmatrix} There is an isomorphism ${\bigotimes}_{i=1}^M (\Mat(N_i), f_i) \stackrel{\sim}{\longrightarrow} (\Mat(N),f)$ of algebras in $\Rep_q(K)$.  \end{prop}

\begin{proof}[Sketch of proof]   By induction, it suffices to prove the case $M = 2$. For each elementary matrix $\elm \in \Mat(N_1)$, define a diagonal matrix $\Delta(\elm) \in \Mat(N_2)$ whose $k$-th diagonal entry is $q^{2\langle f_2(k)  ,\deg(\elm)\rangle},$ for $1 \leq k \leq N_2$. Let  $\otimes_q$ the tensor product in $\Rep_q(K)$. Define a map $$\phi : (\Mat(N_1),f_1) \otimes_1 (\Mat(N_2),f_2) \rightarrow  (\Mat(N_1),f_1)  \otimes_q (\Mat(N_2),f_2)$$ 
by setting $\phi(\elm \otimes Y) = \elm \otimes \Delta(\elm) Y,$ and extending linearly, where $\elm \in \Mat(N_1)$ is an arbitrary elementary matrix. One shows that $\phi$ is a $K$-equivariant isomorphism. The result follows from the case $q=1$, i.e.\ the case of  trivial braiding,  which is elementary. \end{proof}

Suppose $q$ is a primitive $\ell$-th  root of unity. For $1 \leq i \leq M$, write $\overline f_i$  (resp.\ $\overline f$) for the composition of $f_i$ (resp.\ $f$) with the quotient map $\Z^d \rightarrow (\Z/\ell\Z)^d$, which endows $\Mat(N_i)$ (resp.\ $\Mat(N)$) with a $(\Z/\ell\Z)^d$-grading, equivalently, an action of  $\Gamma$. 

\begin{cor} \label{cor:deformmatrix} There is an isomorphism $ {\bigotimes}_{i=1}^M (\Mat(N_i), \overline f_i)  \stackrel{\sim}{\longrightarrow} (\Mat(N),\overline f)$ of algebras in $\Rep_q(\Gamma)$.  \end{cor}

\subsection{Quantum moment maps}\label{subsec:pre:qmm} We recall basic facts about quantum Hamiltonian reduction.

\begin{definition}\label{def:qmm}[Adapted from \cite{VV}.]  Let $H$ be a Hopf algebra and $D$ an algebra in $H\dmod$. An algebra homomorphism $\qmm : H \rightarrow D$ is called a {\it quantum moment map} if, for any $h \in H$ and $a \in D$, the following identity holds: $$\qmm(h)\cdot a = (h_{(1)} \rhd  a) \qmm(h_{(2)}), $$ where we use Sweedler notation $\Delta(h) = h_{(1)} \otimes h_{(2)}$, and  `$\rhd$' denotes the action of $H$ on $D$. \end{definition} 

Let  $\qmm : H \rightarrow D$ be a quantum moment map,  $\eta: H \rightarrow \C$  a character of $H$, and $I_\eta = D(\mu -\eta)$  the left ideal of $D$ generated by the set of all $\qmm(h) - \eta(h)\cdot 1$ for $h\in H$. Then $I_\eta$ is an $H$-submodule, and we have the following standard result \cite[Section 3.4]{BFG}:

\begin{prop}\label{prop:qhamred} The algebra structure on $D$ induces an algebra structure on the invariants $(D/I_\eta)^H$, called the {\rm quantum Hamiltonian reduction} of $D$ by $H$ along $\eta$. \end{prop}

We discuss quantum moment maps in the context of representations of a torus. Any representation $V$ of the torus $K = (\C^\times)^d$ carries the structure of a comodule for the Hopf algebra $\O(K)$:
$$\rho : V \rightarrow V \otimes \O(K) \qquad v \mapsto v \otimes z^\mathbf r, \qquad \text{for} \ v \in V_\mathbf r.$$ A bilinear map $\langle , \rangle : \Z^d \times \Z^d \rightarrow \Z$ induces a map $b: \O(K) \otimes \O(K) \rightarrow \C[t^{\pm 1}]$ taking $(z^\mathbf r, z^\mathbf s)$ to $t^{\langle \mathbf r, \mathbf s \rangle}$. 

\begin{prop}\label{prop:RepKOKmod} Let $q \in \C^\times$. There is an action of $\O(K)$ on $V$ defined by the composition 
$$ V \otimes \O(K) \stackrel{\rho  \otimes \id}{\longrightarrow} V \otimes \O(K) \otimes \O(K) \stackrel{\id \otimes b}{\longrightarrow} V \otimes \C[t^{\pm 1}]  \stackrel{\id \otimes \text{ev}_q}{\longrightarrow} V.$$ We use the symbol `$\rhd$' for this action. Explicitly, $z^\mathbf r \rhd v = q^{\langle \mathbf r, \mathbf s \rangle} v,$ for $v \in V_\mathbf s$. \end{prop}

\begin{comment} Note that the space $V_\mathbf 0$ of invariants of $V$ under the action of $T$ coincides with the space $V^{\O(T)}$ of invariants of $V$ under the action of the Hopf algebra $\O(T)$. 
\end{comment}

Thus we have a faithful tensor functor $B: \Rep(K) \rightarrow \O(K)\dmod$ that commutes with the forgetful functors to vector spaces.

\begin{definition} \label{def:qmmOK} Suppose $D$ is an algebra in $\Rep(K)$.  An algebra homomorphism $\mu : \O(K) \rightarrow D$ is called a {\it  quantum moment map} for the action of $K$ on $D$ if it is a quantum moment map when $D$ is regarded as an algebra in $\O(K)\dmod$ via the functor $B$.   \end{definition} 

\begin{comment} \begin{lemma} The following diagram of functors commutes: \[ \xymatrix{ \Rep(T)  \ar[r]^{B \quad } \ar[d] & \O(T)\text{-mod}  \ar[d] \\ \Rep(K) \ar[r]^{B \quad }& \O(K)\text{-mod} }\] where the left (resp.\ right) vertical map is induced by precomposition with $\phi$ (resp.\ $\phidaggerstar$). \end{lemma} 
\end{comment}

\begin{lemma}\label{lem:phidagger} Suppose $D$ is an algebra in $\Rep(T)$ and $\mu : \O(T) \rightarrow D$ is a quantum moment map. Then the composition $\O(K) \stackrel{\phidaggerstar}{\longrightarrow} \O(T) \stackrel{\mu}{\longrightarrow} D$ is a quantum moment map for the action of $K$ on $D$. \end{lemma}

\begin{rmk} If $D$ is a locally finite $U_q\g$-algebra, then one considers quantum moment maps of the form $\mu :  \O_q(G) \rightarrow D$ in the category of $U_q\g$-modules. When $\g = \frakk$ is the Lie algebra of a torus $K$, the data of such a map is equivalent to the data of a quantum moment map in the sense of Definition \ref{def:qmmOK}, since $U_q\frakk \simeq \O_q(K) \simeq \O(K)$.\end{rmk}

%%%%%%%%%%%%%%%%%%%%%%%%%%%%%%%%%%
\section{An algebra of difference operators}\label{sec:Dqn}

\subsection{The algebra $\Dqn$}\label{subsec:Dqn:def}

In this section, we define an algebra $\Dqc$ of $q$-difference operators on $\C$ and extend the construction to $\C^n$ using braided tensor products. These algebras are variants of the quantum Weyl algebras considered in \cite{GZ}.

\begin{definition}
Let $q \in \C^\times$. The algebra $\Dqc$ of $q$-difference operators on $\C$ is generated by elements $x$ and $\partial$ subject to the relation $\partial x = \qq x \partial + (\qq -1)$. The element $\alpha := 1 + x \partial$ is called the  Euler operator.
\end{definition}

In the following lemma (whose proof is elementary), the first assertion shows that $\Dqc$ is a flat noncommutative deformation of the algebra of functions on the cotangent bundle of $\C$, the second assertion justifies the name `difference operators,' and the third assertion includes the $q$-commutativity property of the Euler operator $\alpha$. 

\begin{lemma} \label{lem:dxn} Let $q\in \C^\times$. \begin{enumerate}
\item A PBW basis for $\Dqc$ is given by the ordered monomials $x^i\partial^j$ for $i,j  \geq0$. Hence, as a vector space, $\Dqc$ is isomorphic to the space $\C[x,\partial]$ of polynomials in two variables.

\item  An action of $\Dqc$ on $\C[t]$ is given by $$ (x \cdot f )(t) := t f(t) , \qquad  (\partial \cdot f)(t) = \frac{f(\qq t) -f(t)}{t}.$$ The Euler operator acts as $(\alpha \cdot f)(t) = f(q^2t)$. This module is isomorphic to the cyclic module $\Dqc/  \Dqc \partial$ under the map $t \mapsto x$. 
 
%Furthermore, we have $$\frac{d f}{d x} = \lim_{q\to 1}\frac{\partial\cdot f}{\qq-1}, \qquad (\alpha\cdot f)(t) = f(\qq t).$$

\item The following identities hold in $\Dqc$, for $i\geq1$: $$\alpha x= \qq x \alpha, \qquad \alpha \partial = q^{-2} \partial \alpha,$$  $$\partial^i x= q^{2i} x \partial^i + (q^{2i} - 1) \partial^{i-1},  \qquad \partial x^i = q^{2i} x^i \partial + (q^{2i} -1) x^{i-1}.$$ 

%\item The algebra $\Dqc$ isomorphic to the skew polynomial algebra\footnote{See \cite{BrownGoodearl} for the definition of a skew polynomial algebra.} $\C[x][\partial; \tau, \delta]$ with $\tau(x) = \qq x$ and $\delta(x^n) = (q^{2n}-1)x^{n-1}$. 
\end{enumerate}\end{lemma}

We now define the algebra $\Dqn$ for $n>1$. Adopt the notation of Section \ref{subsec:pre:tori}, so   $K = (\C^\times)^d$ and $T = (\C^\times)^n$ are standard tori and $\phi: K \hookrightarrow T$ is an inclusion. Fix $q \in \C^\times$. For $i = 1, \dots, n$, let $\Dqc_i$ denote a copy of the algebra $\Dqc$, with generators $x_i, \partial_i$ and Euler operator $\alpha_i = 1 + x_i\partial_i$. Define an action of $T$ on $\Dqc_i$ by $$t \cdot x_i = t_i x_i, \qquad t \cdot \partial_i = t_i\inv\partial_i.$$ Regard each $\Dqc_i$ as an algebra in the braided tensor category $\Rep_q(K)$ (see Section \ref{subsec:pre:braideddeform}) via the map  $\phi$.

\begin{definition}The algebra $\Dqn$ of $q$-difference operators on $\C^n$ is defined as the tensor product of the algebras $\Dqc_i$ in $\Rep_q(K)$: $$\Dqn = \bigotimes_{i=1}^n \Dqc_i.$$ 
The algebra $\U_q = \Dqn^K$ of $K$-invariants is called a hypertoric quantum group. \end{definition}

Recall that the map $\phi$ has the form $\phi(k)_i = \prod_{j=1}^r k_j^{m_{ij}}$ for some $m_{ij} \in \Z$. Define $\deg(i) \in \Z^d$ by $\deg(i)_j = m_{ij}$ and set  $q_{ij} = q^{\langle \deg(j), \deg(i) \rangle}$. The following lemma is straightforward to verify.

\begin{lemma}\label{lem:Dqnrelations} The algebra $\Dqn$ is generated by elements $x_1,\ldots, x_n$, $\partial_1,\ldots, \partial_n$ subject to the following relations for $i <j$:
$$ x_j x_i = q_{ij} x_i x_j, \qquad \partial_j \partial_i = q_{ij} \partial_i \partial_j,$$
$$\partial_j x_i = q_{ij}\inv x_i \partial_j, \qquad x_j \partial_i = q_{ij}\inv \partial_i x_j, \qquad \partial_i x_i = \qq x_i \partial_i + (\qq-1).$$ Consequently, for $i\neq j$, the Euler operator $\alpha_i$ commutes with each of $x_j$ and $\partial_j$. \end{lemma}

\begin{rmk} We make the following remarks:
\begin{enumerate}
\item Setting $q=1$ in the definition of $\Dqn$, we obtain the algebra $\O(T^*\C^n)$ functions on the cotangent bundle of $\C^n$, whereas renormalizing each $\partial_i$ to $\frac{\partial_i}{q^2 -1}$ and setting $q=1$, we obtain the Weyl algebra $\D(\mathbb{C}^n)$ of affine $n$-space. Thus, $\Dqc$ is a $q$-deformation of both the algebra $\O(T^*\C^n)$ and the Weyl algebra $\D(\mathbb{C}^n)$.
 \item The hypertoric enveloping algebra corresponding to the subtorus $K$ of $T$ is the subalgebra of $K$-invariant differential operators on $\mathbb{C}^n$. This algebra was introduced by \cite{MvdB}, and central reductions of this algebra are considered in \cite{BLPW:catO, MvdB}. The algebra $\Dqn^K$ is a quantum version of this algebra; hence the terminology `hypertoric quantum group'.
 \item Observe that $\Dqn$ depends on the inclusion $\phi: K \hookrightarrow T$ and should properly be denoted $\D_q(\C^n; \phi)$. For simplicity, however, we write just $\Dqn$. 
 \item  The fact that $\Dqn$ is an algebra object in $\Rep_q(K)$ reflects the construction of hypertoric varieties (reviewed in Section \ref{subsec:pre:hypertoric}) in that one regards the algebra of functions $\O(T^*\C^n)$ as a representation of the subtorus  $K \subseteq T$.
\end{enumerate}
\end{rmk}

\renewcommand{\Dq}{{\mathcal D}_q}
\subsection{Examples from quivers}\label{subsec:Dqn:quiver}

Let $Q = (\mathcal V, \mathcal E)$ be a finite, loop-free quiver and write  $V = \# \mathcal V$, $E = \# \mathcal E$, and $C = \# \pi_0(Q)$ for the number of vertices, edges, and path components of $Q$. For every edge $e \in \mathcal E$, attach a degree vector $\deg(e) \in \Z^{\mathcal V}$: \[ \deg(e)_v= \left\{ 
  \begin{array}{l l}
    -1 & \quad \text{if $v$ is the source of $e$}\\
    1 & \quad \text{if $v$ is the target of $e$}\\
    0 & \quad \text{otherwise}
  \end{array} \right.\] 
Let $K$ be the quotient of $(\C^\times)^{\mathcal V}$ by the subtorus of elements $g \in (\C^\times)^{\mathcal V}$ such that $g_v =g_w$ if $v$ and $w$ belong to the same path component of $Q$. Thus, $K$ is a torus of rank $V-C$. The pullback of the usual dot product on $\Z^\mathcal V$ under the natural map $X^*(K) \rightarrow X^*((\C^\times)^\mathcal V) = \Z^\mathcal V$ defines a bilinear from on $X^*(K)$; this is the bilinear form we fix in the definition of $\Rep_q(K)$. Each element $\deg(e) \in \Z^\mathcal V$ can be pulled back uniquely to $X^*(K)$. 

There is a well-defined injective homomorphism $\phi: K \rightarrow (\C^\times)^{\mathcal E}$ given by $\phi(k)_e = \prod_{v \in \mathcal V} ({\tilde k}_v)^{\deg(e)_v}$ for any lift $\tilde k \in (\C^\times)^{\mathcal V}$ of $k \in K$. Fix a total ordering on the set $\mathcal E$. To each edge $e \in \mathcal E$, we associate a copy $\Dq(e)$ of the algebra $\Dqc$. 

\begin{definition}\label{def:Dqquiver} The algebra $\Dq(Q)$ of $q$-difference operators associated to $Q$ is defined as the tensor product of the algebras $\Dq(e)$ taken in the category  $\Rep_q(K)$:  $$\Dq(Q) : = \bigotimes_{e \in \mathcal E} \Dq(e).$$ \end{definition}

The algebra of Definition \ref{def:Dqquiver} is isomorphic to the algebra  $\Dq(Q)$ defined in \cite{DJ}, where the dimension vector is equal to one at all vertices. Elements of $\Dq(e)$ commute with elements of $\Dq(f)$ whenever the edges $e$ and $f$ do not share a common vertex.

\begin{example}\label{ex:anquiver} Let $Q$ be the affine $A_{n-1}$ quiver. Thus, $Q$ is a cyclic quiver with $n$ vertices and $n$ edges, each labeled by the set $\{1, \dots, n\}$, where edge $i$ has source vertex $i$ and target vertex $i+1 \mod n$. The algebra $\Dq(Q)$ is generated by elements $x_1,\ldots, x_n$, $\partial_1,\ldots, \partial_n$ subject to the following relations (indices taken modulo $n$):
$$x_{i+1} x_i = q\inv x_i x_{i+1}, \qquad \partial_{i+1} \partial_i = q\inv \partial_i \partial_{i+1},$$
$$\partial_{i+1} x_i = q x_i \partial_{i+1}, \qquad x_{i+1} \partial_i = q \partial_i x_{i+1}, \qquad \partial_i x_i = \qq x_i \partial_i + (\qq-1),$$ and all other pairs of generators commute.  \end{example}

\subsection{The center of $\Dqn$}\label{subsec:Dqn:center}
\renewcommand{\Dq}{\D}

Henceforth, we fix $q\in\C^\times$ to be a primitive $\ell$-th root of unity in $\C$, where $\ell >1$ is odd. Under this hypothesis, the algebra $\Dqn$ develops a large center:

\begin{prop}\label{prop:center} The center $\Zln$ of $\Dqn$ is the subalgebra generated by $x_i^\ell$ and  $\partial_i^\ell$, for $i = 1, \dots, n$ and is isomorphic to a polynomial algebra: 
$$\Zln = \C[x_1^\ell, \dots, x_n^\ell, \partial_1^\ell, \dots, \partial_n^\ell].$$
\end{prop}

\begin{proof} We prove the $n=1$ case first. By Lemma \ref{lem:dxn}(3) and the defining relations of $\Dqc$, the algebra generated by $x^\ell$ and $\partial^\ell$ is contained in the center of $\Dqc$ and is isomorphic to the polynomial algebra $\C[x^\ell, \partial^\ell]$. Conversely, let $z \in Z$. By Lemma \ref{lem:dxn}(1), we can write $z = \sum_{m,k} z_{m,k}x^m \partial^k,$ where the sum ranges over nonnegative integers $m,k$ and $z_{m,k}$ is a scalar. We have that 
\begin{align*}
x z &= \sum_{m,k} z_{m,k} x^{m+1} \partial^k = \sum_{m \geq 1;k\geq 0} z_{m-1,k} x^{m} \partial^k, \qquad \text{and} \\
z x &= \sum_{m,k} q^{2k}z_{m,k} x^{m+1} \partial^k + (q^{2k} -1 ) z_{m,k} x^{m} \partial^{k-1} \\ &= \sum_{m\geq 1; k\geq 0} q^{2k}z_{m-1,k} x^{m} \partial^k + \sum_{m,k}(q^{2(k+1)} -1 ) z_{m,k+1} x^{m} \partial^{k}\end{align*}
Since $x z = z x$, we deduce that\footnote{One also deduces that $(q^{2(k+1)}-1)z_{0, k+1} = 0$ for any $k \geq 0$, but we do not need this.} $z_{m-1, k} = q^{2k} z_{m-1,k} + (q^{2(k+1)} - 1) z_{m, k+1}$
for $m \geq 1$ and $k \geq 0$. From this equation, we obtain
$$(q^{2k} -1) z_{m,k} = (-1)^i (q^{2(k+i)} -1) z_{m+i, k+i}$$
for $m,k,i \geq 0$. Suppose that $\ell$ does not divide $k$, so that $q^{2k} -1 \neq 0$. There is a unique $i \in \{1, \dots, \ell-1\}$ such that $\ell$ divides $k+i$. For this choice of $i$, we have 
$$z_{m,k} = \frac{(-1)^i (q^{2(k+i)} -1) z_{m+i, k+i}}{q^{2k} -1} =  \frac{(-1)^i (0) z_{m+i, k+i}}{q^{2k} -1} = 0.$$
An analogous argument using the identity $\partial z = z\partial$ shows that $z_{m,k} = 0$ whenever $\ell$ does not divide $m$. Thus, $z \in \C[x^\ell, \partial^\ell]$. The general case follows from Lemma \ref{lem:Dqnrelations}, which shows that, for $i \neq j$, the elements  $x_i^\ell$ and $\partial_i^\ell$ commute with each of $x_j$ and $\partial_j$. \end{proof}

\begin{cor}\label{cor:centeretc} We have:
\begin{enumerate}
\item The algebra $\Dqn$ is a free module of rank $\ell^{2n}$ over its center. 
\item\label{cor:alphaell} For $i =1, \dots, n$, the identity  $\alpha_i^\ell = 1 + x_i^\ell \partial_i^\ell$ holds in $\Dqn$.
\end{enumerate}
\end{cor}

\begin{proof} For the first statement, it is straightforward to see that a basis for $\Dqn$ as a $\Zln$-module is given by the ordered monomials $x_1^{m_1} \cdots x_n^{m_n} \partial_1^{k_1} \cdots \partial_n^{k_n}$ for $0 \leq m_i, k_i \leq \ell-1$. For the second assertion, it is enough to show that the identity holds in the $n=1$ case. To this end, a simple induction argument verifies that, for any $m \geq1$:
$$\alpha^m = 1  + \sum_{k=1}^{m-1} c_k^{(m)} x^k \partial^k + q^{m(m-1)} x^m \partial^m,$$
where $c_k^{(m)}$ are certain scalars. Specializing to $m = \ell$ and rearranging terms, we obtain 
$$\sum_{k=1}^{\ell-1} c_k^{(\ell)} x^k \partial^k = \alpha^\ell - x^\ell \partial^\ell - 1$$
The right-hand side is central, so it follows that the left-hand side is also central. By Proposition \ref{prop:center}, we conclude that $c_k^{(\ell)} = 0$ for $k = 1, \dots, \ell-1$. Therefore, $\alpha^\ell = 1 +  x^\ell \partial^\ell.$
\end{proof}

\begin{rmk} We make the following remarks:
\begin{enumerate}
 \item As the center of a noncommutative algebra, the algebra $\Zln$ acquires a Poisson structure given by $\{\partial_i^\ell ,x_j^\ell\}$ $=$ $\delta_i^j \ell \left(1+ x_i^\ell\partial_i^\ell\right).$ Thus, although $\Zln$ and $\C[\{x_i^\ell, \partial_i^\ell\}]$ are isomorphic  as algebras, the Poisson structures differ. One can also compare $\Zln$ to the algebra $\O((T^*\mathbb{A}^N)^{(1)})$ of \cite{BMR, Stadnik}, and others.

\item The scalars $c_k^{(n)}$ appearing in the proof of Corollary \ref{cor:centeretc}.\ref{cor:alphaell} are quantum binomial coefficients, of sorts, as they satisfy the following recursive formula: $c_k^{(n)}$ $=$ $q^{2k} c_k^{(n-1)} + q^{2(k-1)} c_{k-1}^{(n-1)}$.
\end{enumerate}
\end{rmk}

\subsection{The Azumaya locus}\label{subsec:Dqn:azumaya}

In this section, $q$ continues to be a fixed $\ell$-th root of unity, where $\ell>1$ is odd. By the main result of the previous section, $\Dqn$ is a finite-rank module over its center $\Zln$, and thus defines a coherent sheaf of algebras, denoted $\widetilde{\Dqn}$, on $\Spec(\Zln)$. In this section, we compute the Azumaya locus of this sheaf. 

\begin{definition} Let $\Zln^\circ$ be the localization of $\Zln$ along the multiplicative set generated by $\{ \alpha_i^\ell \ | \ i = 1, \dots, n\}$. Define $\Dqn^\circ$ as $\Dqn \otimes_{\Zln} \Zln^\circ$. 
\end{definition}

The $q$-commutativity of the $\alpha_i$ implies that the multiplicative set generated by the $\alpha_i$ satisfies the Ore condition, and the localization $\Dqc[\{\alpha_i\inv \}]$ is isomorphic to $\Dqc^\circ$. The restriction of $\widetilde{\Dqn}$ to $\Spec(\Zln^\circ)$ is the sheaf of algebras $\widetilde{\Dqc^\circ}$ associated to the $\Zln^\circ$-algebra $\Dqn^\circ$.

%We identify $\Spec(\Zln)$ with the cotangent bundle $T^*\C^n$ via Proposition \ref{prop:center}.  Using Propositions \ref{prop:center} and \ref{prop:betaeqns}, 

\begin{theorem}\label{thm:Dqnazumaya} Suppose $q$ is a primitive $\ell$-th root of unity where $\ell >1$ is odd.  The Azumaya locus of $\widetilde{\Dqn}$ is $\Spec(\Zln^\circ)$. In particular, $\Dqn^\circ$ is an Azumaya algebra over $\Zln^\circ$. \end{theorem}

We introduce some notation before the proof of the theorem. Let $\lambda \in \Spec(\Zln)$ be a closed point. Under the isomorphism $\Spec(\Zln) $ $\simeq$ $ T^*\C^n$, we write $\lambda$ in coordinates as $\lambda = (\lambda_i, \lambda_{i^\vee})_{i=1}^n$ for $\lambda_i, \lambda_{i^\vee} \in \C$. We identify the closed points of $\Spec(\Zln^\circ)$ with the subset 
$$\{(\lambda_i, \lambda_{i^\vee}) \in T^*\C^n \ | \ 1 +  \lambda_i  \lambda_{i^\vee} \neq 0 \ \text{for all $i$}\} \subseteq T^*\C.$$

\begin{proof} Suppose that $\lambda \in \Spec(\Zln)$ is a closed point in the Azumaya locus. Thus, the fiber $\Dqn_\lambda$ of $\widetilde{\Dqn}$ over $\lambda$ is a matrix algebra. The $q$-commutativity of the $\alpha_i$ (Lemma \ref{lem:dxn}) implies that each $\alpha_i$ generates a nonzero 2-sided ideal in $\Dqn_\lambda$. As a matrix algebra, $\Dqn_\lambda$ has no nonzero proper ideals. We conclude that each $\alpha_i$ (and hence each $\alpha_i^\ell$) is invertible. Thus, $\lambda$ belongs to $\Spec(\Zln^\circ)$.

Conversely, to show that the Azumaya locus contains $\Spec(\Zln^\circ)$, we first consider the case where $n=1$. Set $\Dq:=\Dqc,$ and $Z:=Z_1 = Z(\Dqc)$. Fix a closed point $\lambda \in \Spec(Z^\circ)$. We may write $\lambda = (c,w)$ for complex numbers $c$ and $w$, with $1 + cw \neq 0$. The fiber of $\widetilde{\Dq}$ at $\lambda$ is $$\Dql = \Dq/ \Dq (x^\ell - c, \partial^\ell -w).$$
When there is no danger of confusion, we use the same notation for elements in $\Dq$ and their images in $\Dql$. Pick an $\ell$-th root of unity $b$ for $c$ and write $\nu = (b,w)$. The quotient $M_{\nu}:=\D_\lambda \slash\D_\lambda (x-b)$ is a left $\D_\lambda$-module of dimension $\ell$. 

We argue that the action map $\rho: \D_\lambda \rightarrow \End(M_\nu)$ is an isomorphism. Since the source and target of $\rho$ have the same dimension, it suffices to show that $\rho$ is surjective. To this end, we will choose a convenient basis for $M_\nu$ and define a collection $\{e_{ij}\}$ of $\ell^2$ elements of $\Dql$ whose images in $\End(M_\nu) \simeq \text{Mat}(\ell)$ are the elementary matrices in our chosen basis.  

Let $\gamma$ be a $\ell$-th root of $1+ cw$. Since $\rho(\alpha^\ell) = 1 + cw$ is a scalar, the endomorphism $\rho(\alpha)$ is diagonalizable. Therefore,  we may choose a basis of $M_\nu$ of eigenvectors $v_i$, with distinct eigenvalues, $q^{2i} \gamma$ for $1 \leq i \leq \ell$. A projector $e_{ii}$ onto the span of $v_i$ is given by:
$$e_{ii} = \prod_{r=1, r \neq i}^k \frac{\alpha - q^{2r} \gamma}{q^{2i}\gamma - q^{2r} \gamma}.$$ It is easy to check that $e_{ii} \cdot v_r = v_r$ if $r = i$ and 0 otherwise. Lemma \ref{lem:dxn} implies that, for $1 \leq i \leq \ell$: \begin{equation} \label{eq:xvi} \text{The operator $\alpha$ acts on $xv_i$ as the scalar $q^{2(i+1)}\gamma$ and on $\partial v_i$ as the scalar $q^{2(i-1)}\gamma$.} \end{equation} In general, the vectors $xv_i$ and $\partial v_i$ may be zero, and hence may fail to be eigenvectors for $\alpha$.

\begin{itemize}
 \item If $c \neq 0$, then the element $x$ is invertible in $\Dql$ with inverse $c\inv x^{\ell-1}$. Consequently, the element $x^{i} v_1 \in M_\nu$ is nonzero for any $i$ and defines an eigenvector for $\alpha$ with eigenvalue $q^{2(i+1)}\gamma$.  Without loss of generality, assume $v_i = x^{i-1} v_1$ for $1\leq i \leq \ell$. The element $e_{ij} := x^{j-i}e_{ii}$ sends $v_i$ to $v_j$ and all other basis vectors to zero. Thus, the image of set $\{e_{ij}\}$ in $\End(M_\nu)$ can be identified with the elementary matrices in the basis $\{v_i\}$. 

 \item Otherwise, if $c = 0$, then  $1 + cw = 1$ and we may take $\gamma = 1$. Identities in Lemma \ref{lem:dxn} imply that one can take $v_i = \partial^{i-1}$. The image of each of the following elements of $\Dql$ can be identified with the corresponding elementary matrix, and they generate $\End(M_\nu)$ as an algebra:
$$ e_{i+1,i} = \partial e_{ii}, \textrm{, for $i = 1, \dots, \ell-1$},\qquad  e_{i-1,i} = (q^{2(1-i)} -1)\inv x e_{ii}\textrm{, for $i = 2, \dots, \ell$}.$$
 \end{itemize}

This completes the proof of the $n=1$ case. For the general case, we introduce the following notation. For $\lambda  = (\lambda_i, \lambda_{i^\vee})_i \in \Spec(\Zln)$, set  $$\Dqn_{\lambda} = \Dqn / \Dqn (\{ x_i^\ell - \lambda_i, \partial^\ell_{i} - \lambda_{i^\vee}\ | \ i = 1, \dots, n \}), \ \text{and}$$ 
$$\Dq(\C)_{i,\lambda} =  \Dqc/ (x_i^\ell - \lambda_i, \partial^\ell_{i} - \lambda_{i^\vee}) \ \text{for $i = 1, \dots, n$}.$$ 
The action of $K$ on each of $\Dqc_i$ and $\Dqn$ descends to an action of $\Gamma$ on $\Dq(\C)_{i, \lambda}$ and $\Dqn_\lambda$, where $\Gamma$ is the finite group defined in Section \ref{subsec:pre:braideddeform}. By Lemma \ref{lem:quot-tensor} there is an isomorphism of algebras in $\Rep_q(\Gamma)$: $$\Dqn_\lambda \simeq \bigotimes_{i=1}^n  \Dq(\C)_{i,\lambda}.$$

Suppose $\lambda \in \Spec(\Zln^\circ)$. Appealing to the proof of the $n=1$ case, there is an algebra isomorphism $\Dqc_{i,\lambda} \stackrel{\sim}{\longrightarrow} \Mat(\ell)$, for each $i$. Moreover, this isomorphism endows $\Mat(\ell)$ with a $(\Z/\ell\Z)^d$-grading (equivalently, a $\Gamma$-action) given by $\deg(\elm_{r,s})_j = (r-s) m_{ij}$. Here $\elm_{r,s}$ is the elementary matrix with the entry 1 in the $r$th row and $s$th column and all other entries equal to 0. Thus, we are in the setting of Section \ref{subsec:pre:deformmatrix} and obtain an isomorphism $\Dqn_{\lambda} \stackrel{\sim}{\longrightarrow} \Mat(\ell^n)$ from Corollary \ref{cor:deformmatrix}. \end{proof}

For notational convenience, we label the rows and columns of a matrix in $\Mat(\ell)$ starting with 0, so the elementary matrices $\elm_{r,s}$ of $\Mat(\ell)$ will be indexed by pairs $(r,s)$ of elements of $\Z/\ell\Z = \{ 0, 1, \dots, \ell-1\}$. Suppose we have an $\ell \times \ell$ matrix algebra $\Mat(\ell)_i$ over $\C$ for  $i=1, \ldots n$, each with elementary matrices denoted by $\elm_{r,s}^i$. We index the rows and columns of a matrix in $\Mat(\ell^n)= \bigotimes_{i=1}^n \Mat(\ell)_i$ each by the set $(\Z/\ell\Z)^n$. An element of this set will be denoted as $\mathbf r = (r_i)_{i=1}^n$ with $0 \leq r_i \leq \ell-1$. Under the  well-known isomorphism of algebras
\begin{equation}\label{eq:tensormatrix} \Mat(\ell)^{\otimes n} =\Mat(\ell)_{1} \otimes \Mat(\ell)_{2} \otimes \cdots\otimes \Mat(\ell)_{n} \stackrel{\sim}{\longrightarrow} \Mat(\ell^n), \end{equation} 
the elementary matrix $\elm_{\mathbf r, \mathbf s} \in \Mat(\ell^n)$ corresponds to the tensor product $\bigotimes_{i=1}^n \elm^i_{r_i, s_i}.$ Tracing through the proof of Theorem \ref{thm:Dqnazumaya}, we obtain the following: 
 
\begin{cor}\label{lem:matalpha} Let $\lambda$ be a closed point in $\Spec(\Zln^\circ)$ and fix $\gamma_i$ to be an $\ell$-th root of $1 + \lambda_i \lambda_{i^\vee}$ for every $i = 1, \dots, n$. Under the isomorphism $\Dqn_\lambda \stackrel{\sim}{\rightarrow} \Mat(\ell^n)$ in the proof of Theorem \ref{thm:Dqnazumaya},  the element $\alpha_i$ corresponds to the diagonal matrix whose $(\mathbf r, \mathbf r)$ entry is given by $\gamma_i q^{-2r_i}$. \end{cor}

\subsection{An \'etale splitting}\label{subsec:Dqn:etale}
\newcommand{\ZLQ}{Z}

As before, we fix $q$ to be a primitive $\ell$-th root of unity, where $\ell >1$ is odd. As shown in the previous section, $\widetilde{\Dqn^\circ}$ is an Azumaya algebra over $\Spec(\Zln^\circ)$. It is not split, since the algebra $\Dqn^\circ$ has no nonzero zero divisors. In this section, we define an \'etale cover of $\Spec(\Zln^\circ)$ that splits the Azumaya algebra $\widetilde{\Dqn^\circ}$.  

Recall that the torus $T = (\C^\times)^n$ acts on $\Dqn$ by $t \cdot x_i = t_i x_i$ and $t \cdot \partial_i = t_i\inv \partial_i$. This action factors to give an action of $\Tell$ on $\Zln$. Observe that $\Spec(\Zln^\circ)$ can be identified, as a $\Tell$-space, with $(T^*\C^n)^{(\ell), \circ}$ from Section \ref{subsec:pre:twisthypertoric}. The following lemma is essentially same as Proposition \ref{prop:TCnmomentmap}.

\begin{lemma} There is a moment map $\mu: \Spec(\Zln^\circ) \rightarrow \Tell$ defined by $\mu(\lambda)_i = 1 + \lambda_i \lambda_{i^\vee}.$ \end{lemma}

The corresponding comoment map is given by $\mu^\#:  \O(\Tell)  = \C[y_i^{\pm \ell}] \rightarrow \Zln^\circ$ taking $y_i^\ell$ to $  \alpha_i^\ell.$ 

\begin{definition} Let $S = \Zln^\circ \otimes_{\O(\Tell)} \O(T)$. \end{definition}

Equivalently, $S$ is the algebra of functions on the fiber product $\Spec(\Zln^\circ) \times_{\Tell} T$, and so there is a Cartesian square  \begin{equation}\label{diag:SpecS} \xymatrix{ \Spec(S) \ar[r]^{ f \quad \ } \ar[d]_{\mu_S} & \Spec(\Zln^\circ) \ar[d]_{\mu} \\
T \ar[r]^{ F }& \Tell }\end{equation} whose maps are described as follows. The closed points of $\Spec(S)$ can be identified with 
$$\{ P= (\lambda, \gamma) \in T^*\C^{n} \times \C^n \ | \ \gamma_i^\ell = 1 + \lambda_i \lambda_{i^\vee} \neq 0 \ \text{for $1 \leq i \leq n$} \},$$ where $\lambda = (\lambda_i, \lambda_{i^\vee})_i \in T^*\C^n$. The left vertical map $\mu_S$ sends $(\lambda, \gamma)$ to $\gamma$, $F$ takes $t$ to $t^\ell$, and the map  $f: \Spec(S) \rightarrow \Spec(\Zln^\circ)$ is induced by the inclusion $\Zln^\circ \hookrightarrow S.$ 

\begin{prop}\label{lem:etalestuff} We have:
\begin{enumerate}
 \item The map $f: \Spec(S) \rightarrow \Spec(\Zln^\circ)$ is \'etale of degree $\ell^n$. 
\item The algebra $S$ is isomorphic to the $\Zln^\circ$-subalgebra $\Zln^\circ[\{\alpha_i\}]$ of $\Dqn^\circ$ generated by the Euler operators $\alpha_i$, for $1 \leq i\leq n$.  In particular, $S$ is an integral domain.
\item  The right $S$-modules $\Dqn^\circ$, $\End_S(\Dqn^\circ)$, and $\Dqn^\circ \otimes_{\Zln^\circ} S$ are finitely generated and projective. 
\end{enumerate}
\end{prop}

Henceforth, using part (1) of the above proposition, we regard $S$ as a commutative (but not central) subalgebra of $\Dqn^\circ$. 

\begin{proof}
The first statement follows from the fact that $f$ is a base-change of the \'etale map $F$. The proof of the second statement is straightforward. 

For the third claim, first note that $\Dqn^\circ \otimes_{\Zln^\circ} S$ is a finitely generated and projective $S$-module since $\Dqn$ is such a module for $Z_n$. To prove that $\Dqn^\circ$ is a finitely generated projective $S$-module, we first recall that, for a finitely generated integral domain over a field, a module is projective if and only if its rank is constant on maximal ideals. Thus, it is enough to show that, for any closed point $P \in \Spec(S)$, the dimension of the fiber $\D_P = \Dqn^\circ \otimes_S k(P)$ as a vector space over $\C$ is $\ell^n$. To this end, write $P = (\lambda, \gamma)$ as above. Observe that $\D_P \simeq \bigotimes_{ i=1}^n \D^{(i)}_P,$ where $$\D^{(i)}_P=\D_q(\C^1)/ \D_q(\C^1)( x_i^\ell - \lambda_i, \partial_i^\ell -\lambda_{i^\vee}, \alpha -\gamma_i).$$ It suffices to show that the dimension of each $\D^{(i)}_P$ is $\ell$. Fix $i$ and write $\D_P := \D^{(i)}_P$ and $ (c,w, \gamma) :=$ $(\lambda_i,\lambda_{i^\vee}, \gamma_i)$. Applications of Lemma \ref{lem:dxn} show that that the following identities hold in $\D^{(i)}_P$  for $1 \leq m \leq \ell$,
$$c \partial^m = \prod_{i=\ell-m+1}^\ell (q^{2i}\gamma -1)x^{\ell-m} \qquad w x^m = \prod_{i=1}^m(q^{2i}\gamma-1)\partial^{\ell-m}.$$ If $\gamma^\ell \neq 1$, then it is clear that $\{1, x, \dots, x^{\ell-1}\}$ and $\{1, \partial, \dots, \partial^{\ell-1}\}$ are each bases for $\D_P$. If $\gamma^\ell=1$, then $\gamma = q^{2k}$ for some $0\leq k\leq \ell-1$ and $\{1, x, \dots, x^{\ell -k-1}, \partial, \dots, \partial^{k}\}$ is a basis for $\D_P$. We conclude that $\Dqn^\circ$ is a finitely generated and projective $S$-module. It follows that $\End_S(\Dqn^\circ)$ is also such a module. \end{proof}

\begin{theorem}\label{thm:Dqnetale} There is an isomorphism of $S$-algebras: $\Dqn^\circ \otimes_{\Zln^\circ} S \stackrel{\sim}{\longrightarrow} \End_{S}(\Dqn^\circ).$ Hence,  $f:\Spec(S) \rightarrow \Spec(\Zln^\circ)$ is an \'etale splitting of $\widetilde{\Dqn^\circ}$.\end{theorem}

\begin{proof} Abbreviate $\Dqn^\circ$ by $\D$ and $\Zln^\circ$ by $\ZLQ$. Define a map $ \Phi : \D \otimes_{\ZLQ} S \rightarrow \End_{S}(\D) $ by $ \Phi(d \otimes s)(d') = dd's$ for any $d,d' \in \D$ and $s \in S$. Since $\Phi$ is a map between finitely generated projective $S$-modules, it suffices\footnote{See, for example, \cite[Chapter III, Lemma 1.11]{Milne:1980bb}.} to show that the map induced on the fibers is an isomorphism for all $P \in \Spec(S)$. For $P = (\lambda,\gamma) \in \Spec(S)$, the fibers are:
$$( \D \otimes_{\ZLQ} S) \otimes_S k(P) = (f^*\D) \otimes_S k(P) =  \D_\lambda$$ 
$$ \End_{S}(\D)   \otimes_S k(P) = \End_{k(P)} ( \D \otimes_S k(P)) = \End_\C(\D_P),$$ The map  $\Phi_P : \D_\lambda \rightarrow \End_\C(\D_P)$ is the natural action map induced by the surjection $\D_\lambda \rightarrow \D_P$. It has already been shown that $\D_\lambda$ isomorphic to the matrix algebra $\Mat(\ell^n)$. The action of $\D_\lambda$ on $\D_P$ is nontrivial, and $\D_P$ has dimension $\ell^n$. Therefore, $\Phi_P$ is an isomorphism. \end{proof}

\begin{cor}\label{prop:Dqnetale} The Azumaya algebra $\widetilde{\Dqn^\circ}$ on $\Spec(\Zln^\circ)$ restricts to a trivial Azumaya algebra along fibers of the moment map $\mu: \Spec(\ZLQ^\circ) \rightarrow \Tell$. More concretely, for any $t \in T$, $$ \Dqn^\circ \otimes_S \O(\mu_S\inv(t))  = \End_{\O(\mu\inv(t^\ell))}( \Dqn/ (\alpha_i - t_i)).$$ \end{cor}

\begin{proof} The restriction of $\widetilde{\Dqn^\circ}$ to the fiber $\mu_S\inv(t)$ over $t$ is the sheaf of algebras corresponding to the $\O(\mu_S\inv(t))$-algebra $\Dqn^\circ \otimes_S \O(\mu_S\inv(t))$. By inspection, this algebra is isomorphic to $\Dqn/ (\alpha_i - t_i)$. The result follows from Theorem \ref{thm:Dqnetale} and Proposition \ref{prop:AzumayaCartesian}. \end{proof}

%%%%%%%%%%%%%%%%%%%%%%%%%%%%

\section{An Azumaya algebra on the multiplicative hypertoric variety}\label{sec:Dqk}

In this section, $q$ is a primitive $\ell$-th root of unity, where $\ell>1$ is odd. We perform a version of quantum Hamiltonian reduction on the algebra $\Dqn$ to produce an Azumaya algebra on the Frobenius-twisted multiplicative hypertoric variety $\Xkl$. 

\subsection{Moment maps}\label{subsec:Dqk:moment} The goal of this section is to construct the vertical maps appearing in the following diagram, and to explain their relation to multiplicative hypertoric varieties. 
\begin{equation}\label{diag:momentmaps}\xymatrix{ \Zln  \ar@{^{(}->}[rr] && \Dqn \ar[rr] && \Mat(\ell^n) \\
\O(\Kell) \ar@{^{(}->}[rr] \ar[u]^{\mu_{q,K}^{(\ell)}} && \O(K) \ar[u]^{\mu_{q,K}} \ar[rr] && \C[\Gamma] \ar[u]^{\mu_\Gamma}}.\end{equation}

Recall the action of $T$ on $\Dqn$ defined in Section \ref{subsec:Dqn:def}, namely, $t \cdot x_i = t_i x_i$ and $t \cdot \partial_i = t_i\inv\partial_i$.  By Proposition \ref{prop:RepKOKmod}, we obtain an action of $\O(T)$ on $\Dqn$. There are induced actions of $K$ and $\O(K)$ on $\Dqn$, and of $\Tell$ and $\Kell$ on $\Zln$. 

\begin{definition}\label{prop:KTZDmomentmaps} Define algebra homomorphisms \begin{align*}  \mu_{q,T} : \O(T) = \C[y_i^{\pm 1}] &\rightarrow \Dqn^\circ  ; &\qquad \mu_{q,K} : \O(K) =  \C[z_j^{\pm 1}] &\rightarrow \Dqn^\circ \\  
y_i &\mapsto \alpha_i & \qquad z_j &\mapsto \prod_{i=1}^n \alpha_i^{ m_{ij}}. \end{align*}   \end{definition}

\newcommand{\mmZTell}{\mu_{\Tell}^\#}
\newcommand{\mmZKell}{\mu_{\Kell}^\#}
\newcommand{\tmmZTell}{\mu_{\Tell}} 
\newcommand{\tmmZKell}{\mu_{\Kell}}

\begin{lemma} We have:
\begin{enumerate}
 \item The homomorphisms $\mu_{q,T}$ and $\mu_{q,K}$ are quantum moment maps.
 
 \item  Taking $\ell$-th powers of $\mu_{q,T}$ and $\mu_{q,K}$, we obtain the homomorphisms: \begin{align*} \mu_{q,T}\eell : \O(\Tell) &\rightarrow \Zln^\circ;&\qquad  \mu_{q,K}\eell : \O(\Kell) &\rightarrow \Zln^\circ .\end{align*}
 
 \item Let $\Gamma = \text{\rm ker}(K \rightarrow \Kell) \simeq (\Z/\ell\Z)^d$ be the finite group considered in Section \ref{subsec:pre:braideddeform}.  The map $\mu_{q,K}$ induces a quantum moment map $\mu_\Gamma : 
\C[\Gamma] \rightarrow \Mat(\ell^n)$ making Diagram \ref{diag:momentmaps} commute. 
\end{enumerate}
 \end{lemma}

\begin{proof} The $q$-commutativity of the $\alpha_i$ implies that $\mu_{q,T}$ is a quantum moment map, and so $\mu_{q,K}$ is a quantum moment map by Lemma \ref{lem:phidagger}. For the second assertion, the fact that $\alpha_i^\ell\in\Zln$ implies that the restriction of $\mu_{q,T}$ to $\O(\Tell)$ sends $z_j^\ell$ to $\alpha_i^\ell \in \Zln$, and similarly for $\mu_{q,K}$. 

We establish the third claim by considering the map on fibers induced by  $\mu_{q,K}$. Fix $\eta \in K$ and $\chi$ a character of $K$. The fiber of $\O(K)$ over $\eta^\ell \in \Kell$ is the quotient of $\O(K)$ by the ideal generated by  $z_j^\ell -\eta_j^\ell$. Thus, this fiber is isomorphic to the group algebra $\C[\Gamma]$. By results in Section \ref{subsec:Dqn:azumaya}, the fiber $\D_\lambda$ of the sheaf $\widetilde{\Dqn}$ over a point $\lambda \in \Spec(\Zln^\circ)$ lying over $\eta^\ell$ is isomorphic to $\Mat(\ell^n),$ and has an action of $\Gamma$. We see that $\mu_{q,K}$ induces the desired quantum moment map $\mu_\Gamma : \C[\Gamma] \rightarrow \Mat(\ell^n)$ making Diagram \ref{diag:momentmaps} commute. \end{proof}

\begin{rmk} Identifying $\O_q(K)$ with $\O(K)$, the $\ell$-center $\O_q\eell(K)$ with $\O(\Kell)$, and $u_q(\Lie(K))$ with $\C[\Gamma]$, we see that the map $\mu_{q,K}$ is a Frobenius quantum moment map in the sense of Definition \ref{def:fqmm}. Diagram \ref{diag:momentmaps} is a special case of Diagram \ref{diag:fqmm}. Similar remarks hold for $T$.  \end{rmk}

The maps $\mu_{q,K}\eell$ and $\mu_{q,T}\eell$ are comoment maps corresponding to torus-valued moment maps\footnote{In Section \ref{subsec:Dqn:etale}, the torus-valued moment map $\mu_{\Tell}$ is denoted simply by $\mu$.}, $$\mu_{\Tell} : \Spec(\Zln^\circ) \rightarrow \Tell ;\qquad \mu_{\Kell} : \Spec(\Zln^\circ) \rightarrow \Kell.$$ Identifying $\Zln^\circ$ with $\O((T^*\C^n)^{(\ell), \circ})$, these coincide with the moment maps of Section \ref{subsec:pre:twisthypertoric}, and we have the following result. 

\begin{lemma}\label{lem:hypertoricdef} Let $\mathcal K = (K, \eta, \chi)$. The Frobenius-twisted multiplicative hypertoric variety $\Xk\eell$ is isomorphic to the GIT quotient of $\tmmZKell\inv(\eta^\ell)$ by $\Kell$: $$\Xk\eell \simeq \tmmZKell \inv(\eta^\ell)  \text{ $\!$/$\! \!$/$\!$}_{\chi\eell}  \Kell.$$\end{lemma}

\subsection{The sheaf $\Dqk$}\label{subsec:Dqk:def}

Fix the data $\mathcal K = (K, \eta, \chi)$, as in Section \ref{subsec:pre:hypertoric}. Thus, $\eta^\ell$ defines a point in $\Kell$. Abbreviate the fiber $\tmmZKell\inv(\eta^\ell)$ of $\tmmZKell$ over $\eta^\ell$ by  $X$. We observe that:

\begin{enumerate}
 \item The variety $X$ is affine with algebra of functions given by $\O(X) = \Zln/(\mmZKell - \eta^\ell)$, that is, the quotient of $\Zln$ by the ideal generated by $\mmZKell(z_j^\ell) - \eta_j^\ell$ for $j = 1, \dots, d$.  
 \item Let $ X^{ss}$ denote the set of  semistable points of $X$ with respect to $\chi\eell$. By Lemma \ref{lem:hypertoricdef}, there is a surjective map $\pi: X^{ss} \rightarrow \Xk\eell$ whose fibers are closed $\Kell$-orbits. 
 \item Consider  $\Dqn/ \Dqn(\mu_{q,K} - \eta)$, that is, the quotient of $\Dqn$ by the left ideal generated by the elements $\mu_{q,K}(z_j) -\eta_j$ for $j= 1, \dots, d$. This quotient is a module for $\O(X)$, and hence defines a coherent sheaf $\widetilde{\D_\eta}$  on $X$.
\end{enumerate}

We arrive at the main definition of this paper:

\begin{definition} The sheaf $\Dqk$ of $q$-difference operators on $\Xk\eell$ is defined as $$\Dqk = \pi_*\left( \widetilde{\D_\eta}  \vert_{X^{ss}}\right)^K.$$   \end{definition} 

Explicitly, for an open subset $U \subseteq \Xkl$, the inverse image $\pi\inv(U)$ is a $K$-invariant subset of $X^{ss}$ and $$\Gamma(U, {\Dqk}) =\Gamma\left(\pi \inv (U), \widetilde{\D_\eta} \right)^K.$$

The following proposition shows that $\Dqk$ is a sheaf of algebras that can be regarded as a quantization of the multiplicative hypertoric variety.  

\begin{prop}\label{prop:Dqkbasics} We have:
\begin{enumerate}
 \item \label{lem:DAalg} The sheaf $\Dqk$ is a sheaf of algebras on $\Xk\eell$.
 \item \label{prop:assocgrad} There is an isomorphism $\Dqk \simeq (\Frl)_* \O_{\Xk}$ of coherent sheaves on $\Xkl$.
\end{enumerate}
\end{prop}
 
\begin{proof} For the first claim, observe that the variety $\Xk\eell$ has a cover by affine open sets $U_r$ $=$ $\Spec\left(\oxr^{\Kell}\right)$ where $r \in \O(X)_{(\chi\eell)^n}$ is $\left(\chi\eell\right)^n$-invariant for $n >0$. Fix one such $r$ and, observing that $\O(X)$ $\simeq$ $\Zln^\circ/ (\mmZKell- \eta^\ell)$, choose a lift $\tilde r \in \Zln^\circ \subseteq \Dqn^\circ$ of $r$. Then $$\Gamma(U_r, \Dqk)  = (\D \otimes_{\O(X)} \oxr)^K = \left(\D[\tilde r\inv]/ \D[\tilde r\inv]( \mu_r- \eta)\right)^K,$$ where we abbreviate $\Dqn^\circ$ by $\D$, and $ \mu_r$ is quantum moment map obtained as the composition of $\mu_{q,K}$ with the localization map $\D \rightarrow \D[\tilde r\inv]$. The algebra structure follows from the general fact stated in Proposition \ref{prop:qhamred}.

For the second claim, let $\mu_K : (T^*\C^n)^\circ \rightarrow K$ be the moment map of Section \ref{subsec:pre:hypertoric}. Now, $\Dqn/ \Dqn(\mu_{q,K} - \eta)$ and $\O(\mu_K\inv(\eta)) = \C[x_i, \partial_i]/(\mu_K^\# - \eta)$ are isomorphic as vector spaces, as modules for $\O(X)$, and as representations of $K$. The result follows from the definitions of $\Dqk$ and $(\Frl)_*\O_{\Xk}$. \end{proof}

\begin{comment}
 It can also be shown that $\Dqk$ is an $\Hell$-equivariant sheaf. 
\end{comment}

Recall that we refer to the algebra $\Dqn^K$ of $K$-invariants of $\Dqn$ as a hypertoric quantum group. If $q$ is a root of unity, then the global sections of $\Dqk$ is a localization of the central reduction  $\U_\eta=  (\Dqn/\Dqn(\mu_{q,K} -\eta))^K$ of the hypertoric quantum group $\Dqn^K$. The following conjecture is a group-valued moment map version of Lemma 3.8 of \cite{McGertyNevins} (see also \cite[Proposition 4.11]{BeKu}).

\begin{conj}\label{conj:globalsections} Suppose that $K$ is unimodular. Then, for a generic choice of character, the global sections of $\Dqk$ is precisely the central reduction $\U_\eta$ of the hypertoric quantum group. 
\end{conj}

\subsection{The Azumaya property}\label{subsec:Dqk:azumaya}

Recall from Section \ref{subsec:pre:hypertoric} that the data $\mathcal K = (K, \eta,\chi)$ is smooth if $K$ acts freely on the $\chi$-semistable points of $\mu_{K} \inv(\eta)$. In this case, the associated multiplicative hypertoric variety $\Xk$ and its Frobenius twist $\Xkl$ are smooth. In this section, we prove the following result:
\begin{theorem}\label{thm:Dqkazumaya} Suppose $\mathcal K$ is smooth. Then the sheaf  $\Dqk$ is an Azumaya algebra over $\Xkl$. \end{theorem}

The proof requires some preliminary results. Let $\phi: K \hookrightarrow T$ be the inclusion of tori from above, specified by the $n$ by $d$ integer matrix $M = (m_{ij})$. The following lemma is easy to verify. 
\begin{lemma}\label{prop:rankker} For any positive integer $\ell$, the transpose of the matrix $M$ induces a map $M^\dagger :$ $(\Z/\ell\Z)^n  \rightarrow (\Z/\ell\Z)^d$ whose kernel is a free $\Z/\ell\Z$-module of rank $n-d$. \end{lemma}

\begin{comment}
\begin{proof} The entries in each column of the matrix $M$ have greatest common divisor equal to one. Therefore, using row and column operations over \Z, the matrix $M$ can be put into the form of the identity matrix in the top $d$ rows and zeros in the bottom $n-d$ matrix. Taking transposed gives the result. \end{proof}
\end{comment}

Let  $\lambda \in \Spec(\Zln^\circ)$ and let $\D_\lambda = \Dqn \otimes_{\Zln} k(\lambda)$ be the fiber of $\Dqn$ over $\lambda$. The action of $K = (\C^\times)^d$ on $\Dqn$ induces an action of $\Gamma = \text{ker}(K \rightarrow \Kell) \simeq (\Z/\ell\Z)^d$ on $\D_\lambda$.

\begin{prop}\label{prop:gammainvariants} There is an isomorphism $(\D_\lambda)^\Gamma \stackrel{\sim}{\longrightarrow} \Mat(\ell^{n-d}) ^{\oplus \ell^{d}}.$ \end{prop}

\begin{proof}  Theorem \ref{thm:Dqnazumaya} gives an isomorphism $\D_\lambda \simeq \Mat(\ell^n)$, and so we have an action of $\Gamma$ on $\Mat(\ell^n)$. Adopt the labeling conventions introduced at the end of Section \ref{subsec:Dqn:azumaya}. The $j$-th generator $q_j \in \Gamma \simeq (\Z/\ell\Z)^d$ acts on $\elm_{\mathbf r, \mathbf s}$ by the scalar $q^{M^\dagger(\mathbf r - \mathbf s)_j}$, so the invariants $\Mat(\ell^n)^\Gamma$ is the subalgebra spanned by the elementary matrices $\{\elm_{\mathbf r, \mathbf s}\}$ such that $M^\dagger (\mathbf r - \mathbf s) = 0$. 

Consider a $\ell^n \times \ell^n$ grid where a star `$*$' occupies each of the squares whose coordinates correspond to an elementary matrix that lies in $\Mat(\ell^n)^\Gamma$, and a zero occupies each of the remaining squares. Then $\Mat(\ell^n)^\Gamma$ is the subspace of all matrices formed by replacing each of the stars with a complex  number, and keeping the zeros in place. The column indices of the starred entries in the $\mathbf r$-th row are in bijection with the kernel of $M^\dagger$:  $$\{ \mathbf s \in (\Z/\ell\Z)^{n} \ | \ M^\dagger( \mathbf r - \mathbf s) = 0 \} = \{ \mathbf s' \in (\Z/\ell\Z)^{n} \ | \ M^\dagger\mathbf s'= 0 \} = \text{ker}(M^\dagger).$$ 
The same statement holds for the starred entries in the $\mathbf r$-th column. The first fact below now follows from Lemma \ref{prop:rankker}, while the second is a simple computation.
\begin{equation}\label{lem:numrowcol} \text{Each row and each column in the grid contains $\ell^{n-d}$ starred entries.} \end{equation}
\begin{equation} \label{lem:enclose} \text{If there is a star in entries $(\mathbf r, \mathbf s)$ and $(\mathbf s, \mathbf t)$, then there is a star in entry $(\mathbf r, \mathbf t)$.}\end{equation}

In the grid, permute the rows so that the $\ell^{n-d}$ starred entries of the first column appear in the first $\ell^{n-d}$ entries. From here, permute the columns so that the first $\ell^{n-d}$ entries of the first row are starred; it is possible to perform such a permutation so that first row stays fixed.  Now the starred entries in the first row and first column enclose a block of size $\ell^{n-d} \times \ell^{n-d}$. By statement \ref{lem:enclose}, this block must be filled in with stars. By statement \ref{lem:numrowcol}, this exhausts the stars in the first  $\ell^{n-d}$ rows and columns. Continue in this way to form  $\ell^{d}$ blocks, each of size $ \ell^{n-d}$. \end{proof}

\begin{comment}
\begin{cor}\label{cor:diagnm} The diagonal squares $(\mathbf n, \mathbf n)$ and $(\mathbf m, \mathbf m)$ of the grid become part of the same matrix algebra if and only if $\Phi  ( \mathbf n - \mathbf m) = 0$. \end{cor}
\end{comment}

Fix $\eta \in K$. Let $\lambda \in X$ and $\gamma \in T$ be such that $\phi^\dagger(\gamma) = \eta$. Then the pairs $P = (\lambda, \eta)$ and $P' =(\lambda, \gamma)$ satisfy $\eta_j^\ell = \prod_{i=1}^n (1 + \lambda_i \lambda_{i^\vee})^{m_{ij}}$ and  $\gamma_i^\ell = 1 + \lambda_i \lambda_{i^\vee}.$ Set 
$$\D_{P} = \D_\lambda/ \D_\lambda(\overline{\mu_K} - \eta) \quad \text{and} \quad \D_{P'} = \D_\lambda/ \D_\lambda(\overline{ \mu_T} - \gamma), $$ where $\overline{\mu_T}$ denotes the composition $\O(T) \stackrel{\mu_{q,K}}{\longrightarrow} \Dqn^\circ \rightarrow \D_\lambda$ and $\overline{\mu_K} = \overline{\mu_T} \circ \phi^{\dagger, \#}$. Each of $\D_P$ and $\D_{P'}$ carries an action of the finite group $\Gamma$ and there is a natural surjection $(\D_{P})^\Gamma \twoheadrightarrow (\D_{P'})^\Gamma$. Observe that $\D_{P}$ is the fiber over $\lambda$ of the sheaf $\Dqn/ \Dqn(\mu_{q,K} - \eta)$ on $X = (\mu_{\Kell})\inv(\eta^\ell) $. 

\begin{rmk} Recall the algebra $S = \Zln^\circ \otimes_{\O(\Tell)} \O(T)$ from Section \ref{subsec:Dqn:etale}. Let $S_K =$ $\Zln^\circ \otimes_{\O(\Kell)} \O(K)$. These are naturally commutative (but not central) subalgebras of $\D=$ $ \Dqn^\circ$. Then $P' = (\lambda, \gamma)$ defines a point in $\Spec(S)$ and $P=(\lambda, \eta)$ defines a point in $\Spec(S_K)$; the latter is the image of $P'$ under the natural map $\Spec(S) \rightarrow \Spec(S_K)$. Observe that  $\D_{P} = \D \otimes_{S_K} k(P)$ and $\D_{P'} = \D \otimes_{S} k(P')$ are the fibers of $\D$ over $P$ and $P'$, respectively. \end{rmk}

The following proposition gives a description of fiberwise Hamiltonian reduction, and is the key technical result of this section.

\begin{prop}\label{prop:fiberHamiltonian} The algebra structure on $\Dqn$ descends the $\Gamma$-invariants $(\D_{P})^\Gamma$. There is an isomorphism of algebras: $(\D_{P})^\Gamma \simeq \Mat(\ell^{n-d}).$ The $\Gamma$-invariants $(\D_{P'})^\Gamma$ of the fiber of $\D_{P'}$ is a vector space of dimension $\ell^{n-d}$ over $\C$. Hence, the surjection $(\D_{P})^\Gamma \twoheadrightarrow (\D_{P'})^\Gamma$ induces an isomorphism $$(\D_{P})^\Gamma \stackrel{\sim}{\longrightarrow} \End_\C( (\D_{P'})^\Gamma).$$\end{prop}   

\begin{proof} Under the isomorphism $\D_\lambda \simeq \Mat(\ell^n)$, the left ideals of $\D_\lambda$ generated by $\{\overline{\mu_{K}}(z_j) - \eta_j\}_{j = 1}^d$ and $\{\overline{\mu_{T}}(y_i) - \gamma_i\}_{i=1}^n$ correspond, respectively, to left ideals  $J$ and $I$ of $\Mat(\ell^n)$. We have  $\D_{P} \simeq \Mat(\ell^n)/ J$ and $\D_{P'} \simeq \Mat(\ell^n)/I.$ Recall from above the fiberwise quantum moment map $\mu_{\Gamma} : \C[\Gamma] \rightarrow \Mat(\ell^n).$ The ideal $J$ of $\Mat(\ell^n)$ is the same as the left ideal generated by $\{\mu_{\Gamma}(q_j) - \eta_j\}_{j=1}^d$. By Proposition \ref{prop:qhamred}, there is an induced algebra structure on $(\D_{P'})^\Gamma \simeq (\Mat(\ell^n)/ J)^\Gamma.$ It suffices to prove that there is an isomorphism of algebras $ \left( \Mat(\ell^n) / J \right)^{\Gamma} \simeq \Mat(\ell^{n-d}) $ and that the dimension of $(\Mat(\ell^n) / I)^{\Gamma}$ over $\C$ is $\ell^{n-d}$. 

By Lemma \ref{lem:matalpha}, $\overline{ \mu_K}(z_j)$ and $\overline{ \mu_T}(y_i)$ each correspond to the diagonal matrix whose $(\mathbf r, \mathbf r)$-entry is $\eta_j q^{-2(M^\dagger\mathbf r)_j}$ and $\gamma_i q^{-2r_i}$, respectively. Thus, the $\mathbf r$-th diagonal entry of $\overline{ \mu_K}(z_j) -\eta_j$ is 0 for all $j = 1, \dots, d$ if and only if $\mathbf r \in \ker(M^\dagger)$. It follows that the left ideal $J$ consists of all matrices in $\Mat(\ell^n)$ whose $\mathbf r$-th column is zero whenever $M^\dagger \mathbf r = 0$. In parallel to the discussion of $\Mat(\ell^n)^\Gamma$ in the proof of Proposition \ref{prop:gammainvariants}, to $J^\Gamma$ we attach a $\ell^n \times \ell^n$ grid with a star in entry $(\mathbf r, \mathbf s)$ if and only if $M^\dagger(\mathbf r - \mathbf s) = 0$ and $M^\dagger\mathbf r \neq 0$. After rearranging the rows and columns as in the last step of the proof of Proposition \ref{prop:gammainvariants}, the stars again form a block diagonal form, except one of the `blocks' is actually zero. Thus $J^\Gamma \stackrel{\sim}{\longrightarrow} \Mat(\ell^{n-d})^{\oplus \ell^{d -1}}$. Since taking $\Gamma$-invariants is exact, we have that $$\left( \Mat(\ell^n) / J \right)^{\Gamma} = \Mat(\ell^n)^\Gamma/ J^\Gamma \simeq \Mat(\ell^{n-d}) ^{\oplus \ell^{d}}/ \Mat(\ell^{n-d}) ^{\oplus \ell^{d-1}}  \simeq \Mat(\ell^{n-d}) $$ 

Similarly, the $\mathbf r$-th diagonal entry of $\overline{ \mu_T}(y_i) -\gamma_i$ is 0 for all $i=1,\dots,n$ if and only if $\mathbf r = \mathbf 0$. It follows that the left ideal $I$ consists of all matrices in $\Mat(\ell^n)$ whose $\mathbf 0$-th column (i.e.\ left-most column) is zero. The $\ell^n \times \ell^n$ grid attached to $I^\Gamma$ has a star in entry $(\mathbf r, \mathbf s)$ if and only if $M^\dagger(\mathbf r - \mathbf s) = 0$ and $\mathbf r \neq 0$. After rearranging the rows and columns, the stars again form a block diagonal form, except the left-most column of the first block is zero. The result that $\dim_\C( (\Mat(\ell^n) / I)^{\Gamma}) = \ell^{n-d}$ now follows from the property that taking $\Gamma$-invariants is exact.\end{proof}

\begin{proof}[Proof of Theorem \ref{thm:Dqkazumaya}] Let $\lambda \in X^{ss}$ and $i: O(\lambda) \hookrightarrow X^{ss}$ be the inclusion of the orbit of the semistable point $\lambda$. The definition of $\Dqk$ implies that the stalk at $\pi(\lambda)$ is isomorphic (as an algebra) to $\Gamma(O(\lambda), i^*(\D/\D(\mu_{q,K} - \eta))\vert_{X^{ss}})^K$. Since the action of $\Kell $ on $O(\lambda)$ is free and transitive, the orbit $O(\lambda)$ is a homogeneous space for $K$ with stabilizer $\Gamma= \text{ker}(K \rightarrow \Kell)$ at the point $\lambda$. Consequently, the $K$-invariant global sections of a $K$-equivariant sheaf on $O(\lambda)$ can be identified  with the $\Gamma$-invariants of the fiber of the sheaf $\Dqn/ \Dqn(\mu_{q,K} - \eta)$ over $\lambda$. This fiber is precisely $\D_{P}$ from above. The result follows from Proposition \ref{prop:fiberHamiltonian}. \end{proof}

\begin{rmk} We make the following remarks:
\begin{enumerate}
 \item The fact that the quantum Hamiltonian reduction of a matrix algebra is again a matrix algebra has been observed in \cite[Proposition  1.5.2]{VV}. However, the rank of the resulting matrix algebra depends on the specific setting; hence our computations in the proofs of Propositions \ref{prop:gammainvariants} and \ref{prop:fiberHamiltonian}. 

 \item  Throughout this section, we have assumed that $\mathcal K$ is smooth, i.e.\ that $\Kell$ acts freely on the semistable points of $X$. This assumption can be weakened as follows. Suppose $\Kell$ acts freely on the stable points $X^\text{s}$ of $X$. Then $\Dqk$ restricts to an Azumaya algebra over the image $\pi(X^{\text{s}})$ of $X^\text{s}$ under the map $\pi: X^{ss} \rightarrow \Xk\eell$. Note that, in this case, $\pi(X^\text{s})$ is a smooth subvariety of $\Xkl$.  
\end{enumerate}
\end{rmk}

\subsection{An \'etale splitting}\label{subsec:Dqk:etale}
\newcommand{\Yk}{Y(\mathcal K)}
\newcommand{\Ykc}{Y(\mathcal K)^\circ}

Suppose $\mathcal K$ is smooth. Recall the moment map $\mu_H : \Xkl \rightarrow \Hell$ from Proposition \ref{prop:frobenius}.

\begin{definition} Let $\Yk = \Xkl \times_{\Hell} H$, so $\Yk$ fits into a Cartesian square:  \begin{equation}\label{eq:sprimez} \xymatrix{  \Yk \ar[r]^{f_\mathcal K} \ar[d]_{\mu_{Y,H}} & \Xkl \ar[d]_{\mu_H} \\  H \ar[r]^{  }& \Hell. }\end{equation} \end{definition}

Since $f_\mathcal K$ is the base change of the \'etale morphism $H \rightarrow \Hell$, it is \'etale of the same degree, namely $\ell^{n-d}$. We devote the remainder of this section to the proof of the following result. 

\begin{theorem}\label{thm:DAetale} The \'etale cover $f_\mathcal K :$  $\Yk\rightarrow \Xkl$ splits the Azumaya algebra $\Dqk$ on $\Xkl$. Moreover, $\Dqk$ is split over fibers of the moment map $\mu_H$. \end{theorem}

Recall the variety $\Spec(S) = \Spec(\Zln^\circ) \times_{\Tell} T$ from Section \ref{subsec:Dqn:etale} and the projection map $\mu_S :$ $\Spec(S) \rightarrow T$. Let $\mu_{S,K} = \phi^\dagger \circ \mu_S$. The variety $\Spec(S)$ acquires an action of $\Kell$ through the action of $\Kell$ on $\Spec(\Zln)$.

\begin{prop}\label{prop:YKstuff} The variety $\Yk$ is isomorphic to the GIT quotient $\mu_{S,K}\inv(\eta)  \text{ $\!$/$\! \!$/$\!$}_{\chi\eell} \Kell$, and admits an open cover by varieties of the form $\Spec(\O(Y)[r\inv]^{\Kell})$ for $r \in \O(X)_{(\chi\eell)^n}$, and $$f_\mathcal K\inv\left(\Spec\left(\O(X)[r\inv]^{\Kell}\right) \right) = \Spec\left(\O(Y)[r\inv]^{\Kell}\right).$$ 
\end{prop}

\begin{proof} Let $Y := \mu_{S,K}\inv(\eta)$ and let $f: Y \rightarrow X$ be the restriction of the map induced by the inclusion $\Zln^\circ \hookrightarrow S$. Observe that there are induced torus-valued moment maps $Y \rightarrow H$ and $X \rightarrow \Hell$. Moreover, $Y  = X \times_{\Hell} H$. The following description of the algebras of functions on $X$ and $Y$ justifies the last assertion: \begin{equation}\label{eq:oxoy} \O(X) \simeq \Zln^\circ/ (\mu_{\Kell}^\# - \eta^\ell), \quad \text{and} \quad \O(Y) \simeq \O(X)[\{y_i^{\pm 1}\}]/ ( y_i^\ell - \alpha_i^\ell, \prod_{i=1}^n y_i^{m_{ij}} - \eta_j). \end{equation} 
The action of $\Kell$ on $\Spec(S)$ is given in terms of the action of $\Kell$ on $\Zln$: $k \cdot (\lambda, \gamma) = (k\cdot\lambda, \gamma)$. It follows that there is a surjective map $X^{ss} \rightarrow Y^{ss}$ and that the  morphism  $f : Y \rightarrow X$ induces a surjection $f_\mathcal K : Y(\mathcal K) \rightarrow \Xk\eell$ that is given locally by $\Spec(\oyr^{\Kell}) \rightarrow \Spec(\oxr^{\Kell})$ for $r \in \O(X)_{(\chi\ell)^n}.$ For any such $r$, it is easy to verify that the following square is Cartesian, where the bottom horizontal map is the $\ell$-th power map, and the left and right vertical maps are induced moment maps,   \begin{equation} \label{diag:fAH} \xymatrix{ \Spec(\oyr^{\Kell}) \ar[rr] \ar[d] && \Spec(\oxr^{\Kell}) \ar[d] \\  H \ar[rr]^{ \Frl }&& \Hell. }\end{equation} The first claim follows. The second statement is an immediate  corollary. 
\end{proof}
 
Let $\pi_Y : Y^{ss} \rightarrow Y(\mathcal K)$ be the quotient map. The algebra $\Dqn^\circ$, regarded as an $S$-module, defines a sheaf $\widetilde{\Dqn^\circ}$ on $\Spec(S)$.

\begin{lemma}\label{lem:dyalocfree}  The sheaf $\Dqyk := \pi_{Y*}( \widetilde{\Dqn} \vert_{Y^{ss}})^K$ is locally free over $\Yk$. \end{lemma}

\begin{proof} Let $P'$ be a closed point of $Y^{ss}$, so $\Kell$ acts freely on the orbit of $P'$ in $Y^{ss}$ (recall that $\mathcal K$ is assumed to be smooth in this section). The same argument as in the proof of Theorem \ref{thm:Dqkazumaya} shows that the fiber of $\Dqyk$ at $\pi_Y(P')$ can be identified with the $\Gamma$-invariants of the fiber of $\D$ over $P'$: $\Dqyk_{\pi_Y(P')} \simeq (\D \otimes_S k(P'))^\Gamma.$ By Proposition \ref{prop:fiberHamiltonian}, this fiber has dimension $\ell^{n-d}$ over $\C$.  Since $\Yk$ is of finite type over $\C$,   a finitely generated coherent sheaf on $Y$ is locally free if and only if its  rank is constant on closed points. The result follows. \end{proof}

\begin{proof}[Proof of Theorem \ref{thm:DAetale}]  We define an isomorphism $\Xi: (f_\mathcal K)^* \Dqk \rightarrow \End_{Y(\mathcal K)}(\Dqyk)$ of sheaves on $Y(\mathcal K)$ as follows. Fix $r \in \O(X)_{(\chi\eell)^n}$ and abbreviate $\Dqn^\circ$ by $\D$. Set 
$$U = \Spec(\oxr^{\Kell}) \subseteq \Xkl, \quad \O(U) = \oxr^{\Kell},$$ $$V = \Spec(\oyr^{\Kell}) \subseteq Y(\mathcal K), \quad \text{and} \quad \O(V) = \oyr^{\Kell}.$$ 
There is an isomorphism of $\O(X)$-algebras $\D/ \D(\mu_{q,K} - \eta) \simeq \D \otimes_S \O(Y)$, which gives rise to an isomorphism of $\O(U)$-algebras
$$ (\D/ \D(\mu_{q,K}  - \eta) \otimes_{\O(X)} \O(X)[r\inv])^K \simeq (\D \otimes_S \O(Y)[r\inv])^K.$$
Left multiplication gives  a map of $\O(U)$-algebras:  $$ (\D/ \D(\mu_{q,K}  - \eta) \otimes_{\O(X)} \O(X)[r\inv])^K \longrightarrow \End_{\O(V)}((\D \otimes_S \O(Y)[r\inv])^K).$$ Tensoring with $\O(V)$, we obtain a map of  $\O(V)$-algebras:
\begin{equation}\label{eq:GITlocalmap} \Xi_V: (\D/ \D(\mu_{q,K}  - \eta) \otimes_{\O(X)} \O(X)[r\inv])^K \otimes_{\O(U)} \O(V) \longrightarrow \End_{\O(V)}((\D \otimes_S \O(Y)[r\inv])^K). \end{equation} 
By Proposition \ref{prop:YKstuff}, $f_\mathcal K\inv(U) = V$, and since $f_\mathcal K$ is surjective, we have that $f_\mathcal K(V) = U.$ Therefore, the source of this map defines the sections of $f_{\mathcal K}^*\Dqk$ over $V$. Similarly, the target of $\Xi_V$ coincides with the sections of $\End_{Y(\mathcal K)} (\Dqyk))$ over $V$. Therefore, $\Xi_V$ defines a morphism of sheaves on $V$. These morphisms are compatible among the members of the affine cover of Proposition \ref{prop:YKstuff}, and we obtain a morphism of sheaves $\Xi: (f_\mathcal K)^*\Dqk \rightarrow \End_{Y(\mathcal K)}(\Dqyk)$ on $Y(\mathcal K)$. 

Let $P' = (\lambda, \gamma) \in Y^{ss}$ and $P = (\lambda, \eta)$ as in Section \ref{subsec:Dqk:azumaya}. The map $\Xi_{\pi_Y(P')}$ induced on fibers is precisely the map $(\D_P)^\Gamma \rightarrow \End_\C( \D_{P'}^\Gamma)$ considered in the last statement of Proposition \ref{prop:fiberHamiltonian}, and hence is an isomorphism. Since $\Xi$ is a map of locally free coherent sheaves, it follows that $\Xi$ is an isomorphism. Proposition \ref{prop:AzumayaCartesian} now implies that $\Dqk$ splits along fibers of the moment map $\mu_H : \Xkl \rightarrow \Hell.$ \end{proof}

\newcommand{\YzK}{Y_0(\mathcal K)}

Let $\MzK\eell = \frakM(K, 0, \eta)\eell = \Spec(\O(X)^{\Kell})$ and $Y_0(\mathcal K)= Y(K, 0,\eta) = \Spec(\O(Y)^{\Kell})$ be the affinizations of $\Xkl$ and $\Yk$. There is a map $f_{0, \mathcal K} : \YzK \rightarrow \MzK\eell$ making the following diagram commute, where the vertical maps are the affinization maps: \[\label{diag:affinizations} \xymatrix{  \Yk \ar[r]^{f_\mathcal K} \ar[d] & \Xkl \ar[d]^{\nu\eell} \\  \YzK \ar[r]^{f_{0,\mathcal K}} & \MzK\ell } \]

\begin{prop} The diagram above is Cartesian, and the Azumaya algebra $\Dqk$ splits over fibers of the affinization map $\nu\eell: \Xkl \rightarrow \MzK\eell.$ \end{prop}

\begin{proof} Recall the descriptions of the algebras $\O(X)$ and  $\O(Y)$ in \ref{eq:oxoy}. Observe that each $y_i \in \O(Y)$ is $\Kell$-invariant. Thus, $\O(Y)^{\Kell} = \O(X)^{\Kell} \otimes_{\O(\Hell)} \O(H).$ Fix $r \in \O(X)_{\chi^n}$ for $n >1$. Then 
$$\O(X)[r\inv]^{\Kell} \otimes _{\O(X)^{\Kell}} \O(Y)^{\Kell} = \O(X)[r\inv]^{\Kell} \otimes_{\O(\Hell)} \O(H).$$ By the proof of Proposition \ref{prop:YKstuff} (specifically, diagram \ref{diag:fAH}), the right-hand side is isomorphic to $\O(Y)[r\inv]^{\Kell}$, so square above is Cartesian. The last statement follows from Proposition \ref{prop:AzumayaCartesian}. \end{proof}

\subsection{Quiver varieties}\label{subsec:Dqk:quiver}

Let $Q=(\mathcal V, \mathcal E)$ be a finite, loop-free quiver as is in Section \ref{subsec:Dqn:quiver}. Fix dimension vector $\mathbf d$ equal to one at all vertices. This data gives rise to a multiplicative quiver variety (see \cite{C-BS, DJ}), which is in fact a multiplicative hypertoric variety. The algebra $$\U_\eta = (\D_q(Q)/\D_q(Q)(\mu_{q,K}  -\eta))^K$$ coincides with the quantized multiplicative quiver variety $A_{\mathbf d}^\eta(Q)$ constructed in \cite{DJ}. We note that Boalch  has introduced a more general notion of a multiplicative quiver variety, see \cite{Boalch}.

\begin{rmk} For quiver varieties, the subtorus $K \subseteq (\C^\times)^\mathcal E$ is unimodular, so the resulting multiplicative hypertoric variety is smooth and affine for regular values of the moment map.  \end{rmk}

\begin{example} Let $Q$ be the affine  $A_{n-1}$ quiver considered in Example \ref{ex:anquiver}. The central reduction  $\U_1 = \D_q(Q)^K/ \D_q(Q)^K (\mu_{q,K}  - 1)$ of the hypertoric quantum group $\D_q(Q)^K$ is generated by elements $A^{\pm 1}, B, C$ subject to the relations:
$$ A B = q^2 B A, \qquad  A C = q^{-2} C A$$
$$BC = q^{\frac{n(n-1)}{2}} (A -1)^n, \qquad CB = q^{\frac{n(n-1)}{2}} (q^2 A -1)^n.$$
The algebra $\U_1$ has an action on $\C[t]$ by difference operators given by:
$$ (A \cdot f)(t) = f(q^2t) \qquad  (B \cdot f)(t) =  t f(t)$$
$$ (C \cdot f)(t) = q^{\frac{n(n-1)}{2}} \frac{\sum_{k=0}^n (-1)^{n-k} {n \choose k} f(q^{2k} t)}{t}$$

Suppose $q$ is a primitive $\ell$-th root of unity, where $\ell > 1$ is odd. Define a commutative algebra $Z$ with generators $a^{\pm 1}$, $b$, and $c$, subject to the single relation $bc = (a - 1)^n$. One checks that:

\begin{enumerate}
 \item The algebra $Z$ embeds as a central subalgebra of $\U_1$ via the map $a \mapsto A^\ell$, $b \mapsto B^\ell$, and $c \mapsto C^\ell$.
 \item The algebra $Z$ coincides with the algebra of global functions on the affine multiplicative quiver variety for $Q$ with moment map parameter 1: $\cM_0 = {\mathfrak M}(K,1,0)$.
 \item The algebra $Z$ has an \'etale cover $\tilde Z$ generated by elements $T$, $b$, and $c$, subject to the relation $bc = (T^\ell - 1)^n$. The map from $Z$ to $\tilde Z$ sends $a$ to $T^\ell$.
\end{enumerate}

Let $\cM = {\mathfrak M}(K, 1, \chi)$ be the multiplicative quiver variety for $Q$ with moment map parameter 1 and character $\chi$ chosen so that $\cM$ is smooth. Assuming $\chi$ can be chosen so that Conjecture \ref{conj:globalsections} holds, the global sections of $\cM$ can be identified with $Z$ (as opposed to a localization thereof). Moreover, there is an Azumaya algebra $\cA$ on $\cM$ whose global sections is $\U_1$.  Let $\tilde\cM = Y(K, 1, \chi)$ be the \'etale cover of $\cM$, defined via the construction of Section \ref{subsec:Dqk:etale}.  The natural map $\tilde \cM \rightarrow \cM$ is an \'etale splitting of $\cA$, and the algebra of global functions on $\tilde \cM$ can be identified with $\tilde Z$. 
\end{example}

\subsection{A localization theorem}\label{subsec:Dqk:localization}
\newcommand{\projaffine}{\nu}
\newcommand{\X}{\mathfrak M}
\newcommand{\globalsec}{A_\eta}

For an abelian category $\A$, write $D(\A)$ and $D^b(\A)$ for its derived category and bounded derived category, respectively. Let $\mathcal T$ be a triangulated category with infinite coproducts, for example $\mathcal T = D(\A)$. An object $P$ of $\mathcal T$ is {\it  compact} if the functor $\Hom_{\mathcal T}(P, - )$ preserves sums. The subcategory $\langle P \rangle$ generated by a compact object $P$ is the full subcategory of $\mathcal T$ consisting of objects $M$ such that $\Hom_{\mathcal T}(P[n], M) \neq 0$ for some $n \in \Z$ whenever $M \neq 0$. 

Throughout this section, assume that $\mathcal K$ is smooth. Let $\globalsec$ be the algebra of global sections of $\Dqk$. As remarked in Section \ref{subsec:Dqk:def}, this algebra is a localization of a central reduction $\U_\eta = \left(\Dqn/\Dqn(\mu_{q,K} - \eta)\right)^K$ of a hypertoric quantum group, and conjecturally we frequently have an isomorphism $A_\eta \simeq \U_\eta$. Let $\globalsec\dmod$ denote the  category of finitely-generated $\globalsec$-modules. The  category of locally finitely generated (coherent) $\Dqk$-modules is denoted by  $\Dqk\dmod$. Since $\Xkl$ is a finite-dimensional variety, the functor of global sections $$\Gamma : \Dqk\dmod \rightarrow \globalsec\dmod$$ has finite homological dimension, and so the derived functor $R\Gamma$ descends to the bounded derived categories,  $R\Gamma : D^b(\Dqk\dmod) \rightarrow D^b(\globalsec\dmod)$. The localization functor $\Loc : \globalsec\dmod \rightarrow \Dqk\dmod$, defined by $V \mapsto   \Dqk \otimes_{\globalsec}  V$, is left adjoint to $\Gamma$. Let $\mathcal L  : D^b(\globalsec\dmod) \rightarrow D(\Dqk\dmod)$ denote the left derived functor of $\Loc$.

\begin{theorem}\label{thm:localization} Assume $\mathcal K$ is smooth. 
\begin{enumerate}
\item If $\eta$ is a regular value of the moment map $\mu_{K}$, then the functor $\Gamma$ defines an equivalence of abelian categories: $$\Gamma : \Dqk\dmod \stackrel{\sim}{\longrightarrow} \globalsec\dmod.$$

\item If the functor $\Loc$ has finite homological dimension, then $R\Gamma$ defines an equivalence of triangulated categories: $$R\Gamma: D^b\left(\Dqk\dmod\right) \stackrel{\sim}{\longrightarrow} D^b\left(\globalsec\dmod\right).$$ %whose inverse if $ \mathcal L.$ 
\item For any $\eta$, then $R\Gamma$ defines an equivalence of triangulated categories: $$ R\Gamma : \left\langle \Dqk \right\rangle \stackrel{\sim}{\longrightarrow} D(\globalsec\dmod),$$ where $\langle \Dqk \rangle$ is the subcategory of $D(\Dqk\dmod)$ generated by the compact object $\Dqk$. 
\end{enumerate}
\end{theorem}

\begin{proof} Abbreviate $\Xk$ by ${\X}$, $\Xkl$ by ${\X}\eell$, and $\Dqk$ by $\D$. If $\eta$ is a regular value of the moment map $\mu_{K}$, then the varieties $\X$ and $\X\eell$ are affine, and the first statement follows. 

Let $\projaffine: \X \rightarrow \X_0 := \X(K, \eta,0)$ be the map to the affine quotient (note the lack of $(\ell)$). Since $\projaffine$ is birational, the relative Grauert-Riemenschneider theorem implies that  $R^i \projaffine_* \omega_{\X} = 0$ for $i >0$, where $\omega_{\X}$ is the dualizing sheaf \cite[p.\ 59]{Esnault:1992aa}. Since $\X$ is symplectic, $\omega_{\X} = \O_{\X}$, and we have that $R^i \projaffine_*  \O_{\X} = 0$ for $i >0$. Recall from Proposition \ref{prop:frobenius} that the Frobenius map $F_\ell : \X \rightarrow \X\eell$ is finite, so has no higher direct images. Using these observations, together with Proposition \ref{prop:assocgrad} and the fact that ${\X}_0$ is an affine scheme, we see that the sheaf $\D$ has no higher global sections:
\begin{align*} R\Gamma({\X}\eell, \D) &= R\Gamma({\X}\eell, F_*\O_{\X}) = R\Gamma({\X},\O_{\X}) = R\Gamma({\X}_0, \projaffine_*\O_{\X})  = \Gamma({\X}_0, \projaffine_*\O_{\X}) \\ &= \Gamma({\X}, \O_{\X}) = \Gamma({\X}\eell, F_*\O_{\X}) = \Gamma({\X}\eell, \D) = \globalsec. \end{align*} 
The proof of the second statement now follows from the methods developed in \cite{BMR} and \cite{BeKa}, specifically, Lemmas 3.5.2 and 3.5.3 of \cite{BMR}. 

We claim that the functor $R\Gamma$ is representable by $\D$. Indeed, for $\mathcal N \in D(\D\dmod)$, we have $$\Hom_{D(\D\dmod)} ( \D, \mathcal N) = R\Gamma(\X\eell, R\mathcal{H}{\rm om}_{\D}(\D, \mathcal N)) = R\Gamma(\X\eell, \mathcal N).$$ In particular, the endomorphism algebra of $\D$ is  $R\Gamma(\X\eell, \D) = \globalsec.$ Since  $\D$ is a compact generator of $\langle\D\rangle$ and $\globalsec$ is a compact generator of $D(\globalsec\dmod)$, the third statement is a consequence of general results from Morita theory for triangulated categories, see  e.g.\ \cite[Proposition 3.10]{Schwede}. \end{proof}

\begin{rmk} The first statement of the theorem is true for generic values of $\eta$ since regular values of the moment map are dense. An interesting problem is to provide explicit conditions, in terms of $\eta$ and $\chi$, for when the functor $\Loc$ has finite homological dimension. It is expected that a version of the conjectures and examples in \cite[4.4.8-4.4.11]{Stadnik} hold in the multiplicative setting. \end{rmk}

%%%%%%%%%%%%%%%%%%%%%%%%%%%%%%

%\addcontentsline{toc}{section}{References}
\bibliographystyle{plain}

\begin{thebibliography}{bib}

\bibitem[AMM]{AMM} A.~Alekseev, A.~Malkin, E.~Meinrenken. 
\newblock Lie group valued moment maps.
\newblock {\it Journal of Differential Geometry},
\newblock 48 (1998), 445--495. 

\bibitem[BaKr]{BaKr}
E.~Backelin, K.~Kremnitzer.
\newblock Localization for quantum groups at a root of unity.
\newblock {\it Journal of the AMS},
\newblock 21 (2008), 1001--1018.

\bibitem[BDMN]{BDMN}
G.~Bellamy, C.~Dodd, K.~McGerty, T.~Nevins.
\newblock Categorical Cell Decomposition of Quantized Symplectic Algebraic Varieties.
\newblock Preprint: ar{X}iv:1311.6804, 2014. To appear in {\it Geometry \& Topology}.

\bibitem[BeKu]{BeKu}
G.~Bellamy, T.~Kuwabara.
\newblock On Deformation Quantizations of Hypertoric Varieties.
\newblock {\it Pacific Journal of Mathematics},
\newblock 260 (2012), no.\ 1, 89--127. 

\bibitem[BeKa]{BeKa}
R.~Bezrukavnikov, D.~Kaledin.
\newblock McKay equivalence for symplectic quotient singularities.
\newblock {\it Proc. of the Steklov Inst. of Math.}
\newblock 246 (2004), 13-33.

\bibitem[BMR]{BMR}
R.~Bezrukavnikov, I.~Mirkovi{\'c}, and D.~Rumynin.
\newblock Localization of modules for a semisimple Lie algebra in prime characteristic.
\newblock {\em Annals of Mathematics},
\newblock 167 (2008), 945--991.

\bibitem[BFG]{BFG}
R.~Bezrukavnikov, M.~Finkelberg, V.~Ginzburg.
\newblock Cherednik algebras and Hilbert schemes in characteristic $p$.
\newblock {\em Representation Theory},
\newblock 10 (2006), 254-298.
\newblock  With an appendix by Pavel Etingof.

\bibitem[BD]{BD}
R.~Bielawski, A.~Dancer.
\newblock The geometry and topology of toric hyperkaehler manifolds.
\newblock {\it Comm. Anal. Geom.}
\newblock 8 (2000), 727--759. 

\bibitem[Bo]{Boalch}
P.~Boalch.
\newblock Global Weyl groups and a new theory of multiplicative quiver varieties.
\newblock {\it Geometry \& Topology} 19 (2015) 3467--3536. 

\bibitem[BLPW12]{BLPW:catO}
T.~Braden, A.~Licata, N.~Proudfoot, B.~Webster.
\newblock Hypertoric category $\mathcal{O}$ 
\newblock {\it Adv.\ Math}.
\newblock 231 (2012) 1487--1545.

\bibitem[BLPW14]{BLPW:symres}
T.~Braden, A.~Licata, N.~Proudfoot, B.~Webster. 
\newblock  Quantizations of conical symplectic resolutions II: category $\mathcal O$ and symplectic duality. 
\newblock Preprint: ar{X}iv:1407.0964. 2014. To appear in {\it Ast\'erisque}.

\bibitem[BPW]{BPW}
T.~Braden, N.~Proudfoot, B.~Webster. 
\newblock Quantizations of conical symplectic resolutions I: local and global structure.
\newblock Preprint: ar{X}iv:1208.3863. 2014. To appear in {\it Ast\'erisque}.

\bibitem[BG]{BrownGoodearl}
K.A.~Brown, K.R.~Goodearl.
\newblock Lectures on Algebraic Quantum Groups.
\newblock{\em Springer Verlag}, 2002.

\bibitem[C-BS]{C-BS}
W.~Crawley-Boevey, P.~Shaw.
\newblock Multiplicative preprojective algebras, middle convolution and the  {D}eligne-{S}impson problem.
\newblock {\em Adv. Math}, 201 (2006), 180--208.

\bibitem[EV]{Esnault:1992aa} H.~Esnault, E.~Viehweg. 
\newblock Lectures on Vanishing Theorems. 
\newblock {\em Birkh{\"a}user Verlag}, {O}berwolfach Seminars, Vol. 20, 1992

\bibitem[GZ]{GZ}
A.~Giaquinto, J.~Zhang.
\newblock Quantum Weyl algebras.
\newblock {\em J. Algebra}, 176 (1995), 861--881.

%\bibitem[H]{Hartshorne:1977qf}
%R.~Hartshorne.
%\newblock {\em Algebraic Geometry}.
%\newblock Springer, 1977.

\bibitem[HS]{HauselStrumfels}
T.~Hausel, B.~Sturmfels.
\newblock Toric hyperkaehler varieties.
\newblock {\it Documenta Mathematica},
\newblock 7 (2002), 495-534.

\bibitem[J]{DJ}
D.~Jordan.
\newblock Quantized multiplicative quiver varieties.
\newblock {\it Advances in Mathematics},
\newblock 250 (2014), 420--466.

%\bibitem{Kassel:1995uq} C.~Kassel. 
%\newblock {\em Quantum Groups}. 
%\newblock Springer-Verlag, 1995.

\bibitem[LY]{LY}
J.~Levitt, M.~Yakimov.
\newblock Quantized Weyl algebras at roots of unity
\newblock Preprint: ar{X}iv:1606.02121, 2016.

\bibitem[MS]{McBreenShenfeld}
M.~McBreen, D.~Shenfeld.
\newblock Quantum cohomology of hypertoric varieties.
\newblock {\it Letters in Mathematical Physics},
\newblock 103 (2013), 11, 1273--1291.

\bibitem[MN]{McGertyNevins}
K.~McGerty, T.~Nevins.
\newblock Derived equivalence for quantum symplectic resolutions.
\newblock {\it Selecta Mathematica},
\newblock 20 (2014), 2, 675--717.

\bibitem[Mi]{Milne:1980bb}
J.~S.~Milne.
\newblock {\em{\'E}tale {C}ohomology}.
\newblock Princeton University Press, 1980.

%\bibitem[Mu]{Mukai:2003aa}
%S.~Mukai.
%\newblock {\em An Introduction to Invariants and Moduli}.
%\newblock Cambridge University Press, 2003.

\bibitem[MvdB]{MvdB}
I.M.~Musson, M.~van den Bergh.
\newblock {\em Invariants Under Tori of Rings of Differential Operators and Related Topics}.
\newblock American Mathematical Society, 1998.

\bibitem[Sa]{Saltman}
D.~Saltman.
\newblock {\em Lectures on Division Algebras}.
\newblock Issue 94 of Regional Conference Series. American Mathematical  Society, 1999.

\bibitem[Sc]{Schwede}
S.~Schwede.
\newblock {\em Morita theory in abelian, derived and stable model categories},
\newblock in {\em Structured Ring Spectra}, 33--86, London Math. Soc. Lecture Notes 315,
Cambridge Univ. Press, 2004.

\bibitem[St]{Stadnik}
T.~J. Stadnik.
\newblock {\'E}tale splittings of certain {A}zumaya algebras on toric and
  hypertoric varieties in positive characteristic.
\newblock {\it Advances in Mathematics},
\newblock 233 (2013), 268--290.

\bibitem[VV]{VV} 
M.~Varagnolo, E.~Vasserot.
\newblock Double affine Hecke algebras at roots of unity.
\newblock {\it Representation Theory (AMS)},
\newblock 14 (2006),  no.\ 15, 510--510.

\end{thebibliography}

\end{document}